\documentclass[11pt]{article}
\usepackage{amsmath,amssymb,amsthm,fancyhdr,graphicx,pstricks,comment,hyperref,pdflscape}
\usepackage[overload]{empheq} 
\usepackage[all]{xy}
\title{On SL(3,$\mathbb{C}$)-representations of the Whitehead link group}
\author{Antonin Guilloux \\ IMJ-PRG, UPMC \\
OURAGAN, INRIA \\
4, Place Jussieu\\
75005 Paris, France
\and Pierre WILL \\
Universit\'e Grenoble Alpes\\
Institut Fourier, B.P.74\\
38402 Saint-Martin-d'H\`eres Cedex\\
France}
\begin{document}
\maketitle

\def\og{\leavevmode\raise.3ex\hbox{$\scriptscriptstyle\langle\!\langle$~}}
\def\fg{\leavevmode\raise.3ex\hbox{~$\!\scriptscriptstyle\,\rangle\!\rangle$}}
\newcommand{\la}{\langle}
\newcommand{\ra}{\rangle}
\newcommand{\HdC}{\mathbf H^{2}_{\mathbb C}}
\newcommand{\HdR}{\mathbf H^{2}_{\mathbb R}}
\newcommand{\HtR}{\mathbf H^{3}_{\mathbb R}}
\newcommand{\HuC}{\mathbf H^{1}_{\mathbb C}}
\newcommand{\HnC}{\mathbf H^{n}_{\mathbb C}}
\newcommand{\Cdu}{\mathbb C^{2,1}}
\newcommand{\Ct}{\mathbb C^{3}}
\newcommand{\C}{\mathbb C}
\newcommand{\PP}{\mathbb P}
\newcommand{\Mn}{\mbox{M}_n\left(\mathbb C\right)}
\newcommand{\R}{\mathbb R}
\newcommand{\A}{\mathbb A}
\newcommand{\cC}{\mathfrak C}
\newcommand{\cCp}{\mathfrak C_{\mbox{p}}}
\newcommand{\cT}{\mathfrak T}
\newcommand{\rR}{\mathfrak R}
\newcommand{\hH}{\mathfrak H}
\newcommand{\n}{\noindent}
\newcommand{\p}{{\bf p}}
\newcommand{\lox}{\mbox{\scriptsize{lox}}}
\newcommand{\Td}{T^{2}_{\left(1,1\right)}}
\newcommand{\T}{T_{\left(1,1\right)}}
\newcommand{\PSLdR}{\mbox{\rm PSL(2,$\R$)}}
\newcommand{\PSLdC}{\mbox{\rm PSL(2,$\C$)}}
\newcommand{\SLdC}{\mbox{\rm SL(2,$\C$)}}
\newcommand{\SLdR}{\mbox{\rm SL(2,$\R$)}}
\newcommand{\Pu}{\mbox{\rm PU(2,1)}}
\newcommand{\X}{\mbox{\textbf{X}}}
\newcommand{\tr}{\mbox{\rm tr}}
\newcommand{\B}{\mathcal{B}}
\newcommand{\bx}{{\boxtimes}}
\newcommand{\bp}{{\bf p}}
\newcommand{\om}{\omega}
\newcommand{\bm}{{\bf m}}
\newcommand{\bv}{{\bf v}}
\newcommand{\bA}{{\bf A}}
\newcommand{\bB}{{\bf B}}
\newcommand{\bc}{{\bf c}}
\newcommand{\ba}{{\bf a}}
\newcommand{\bd}{{\bf d}}
\newcommand{\bb}{{\bf b}}
\newcommand{\bn}{{\bf n}}
\newcommand{\bq}{{\bf q}}
\newcommand{\bz}{{\bf z}}
\newcommand{\td}{{\tt d}}
\newcommand{\PGL}{{\mathrm{PGL}}}
\newcommand{\SL}{{\mathrm{SL}}}
\newcommand{\PSL}{{\mathrm{PSL}}}
\newcommand{\Hom}{\mbox{\rm Hom}}
\newcommand{\tD}{{\tt D}}
\newcommand{\I}{\mathcal{I}}
\newcommand{\Z}{{\mathbb Z}}
\newcommand{\Isom}{\mathcal{I}}
\renewcommand{\arg}{\mbox{arg}}
\renewcommand{\Re}{\mbox{\rm Re}\,}
\renewcommand{\Im}{\mbox{\rm Im}\,}
\renewcommand{\leq}{\leqslant}
\renewcommand{\geq}{\geqslant} 
\newtheorem{theo}{Theorem}
\newtheorem*{theo*}{Theorem}
\newtheorem{lem}[theo]{Lemma}
\newtheorem{coro}[theo]{Corollary}
\newtheorem{prop}[theo]{Proposition}
\newtheorem{conj}[theo]{Conjecture}
\theoremstyle{definition}
\newtheorem{defi}{Definition}
\newtheorem{rem}{Remark}
\newtheorem{ex}{Example}

\begin{abstract}
{We describe a family of representations in SL(3,$\C$) of the fundamental group $\pi$ of the 
Whitehead link complement. These representations are obtained by considering pairs of regular order 
three elements in SL(3,$\C$) and can be seen as factoring through a quotient of $\pi$ defined by a 
certain exceptional Dehn surgery on the Whitehead link. Our main result is that these representations form an algebraic component of the SL(3,$\C$)-character variety of $\pi$.}
\end{abstract}
%{\small AMS classification 51M10, 32M15, 22E40}

\section{Introduction}

 Let $M$ be a manifold. The description of the character variety 
of $\pi_1(M)$ in a Lie group $G$ is closely related to the study
of geometric structures on $M$ modelled on a $G$-space $X$. In this setting,
representations of $\pi_1(M)$ into $G$ appear as holonomies of $(X,G)$-structures.
In the case of a hyperbolic $3$-manifold $M$, a natural target group  is $\PSLdC$ (or $\SLdC$), 
as the holonomy of a hyperbolic structure on $M$ has image contained in $\PSLdC$. The study of these character varieties
was initiated by Thurston, in the non-compact case, who described a natural way of constructing explicit representations of 
$\pi_1(M)$ in $\PSLdC$ using ideal triangulations of $M$ (see \cite{Thu}). The rough idea is to parametrise  hyperbolic 
ideal tetrahedra using cross-ratios, and to analyse the possible ways of constructing the hyperbolic structure on $M$ by 
gluing together these ideal tetrahedra. This method gives rise to a family of polynomial equations expressed in terms of a 
family of cross-ratios, which are often referred to as {\it Thurston's gluing equations} (see Chapter 4 of \cite{Thu}). The
output of this method is a subvariety of $\C^n$ consisting of those tuples of parameters that satisfy Thurston's equations, which 
is called the {\it deformation variety}. Representations can be expressed in terms of the cross-ratios, and one of 
the main interests of the deformation variety is that it allows explicit computations, which are very useful 
for experiments.

Thurston's approach has been generalized for the  higher dimensional target groups $\SL(n,\C)$ in 
\cite{BFG,GTZ,DGG,GGZ,Zick}. This generalisation is geometrically meaningful. Indeed, 
the subgroups SU$(2,1)$ and $\SL(3,\R)$ of $\SL(3,\C)$ correspond respectively to spherical 
CR structures (see below) and real projective flag structures (see \cite{FS}), whereas 
$\SL(4,\R)$ corresponds to projective structures on $3$-manifolds, well-studied in the convex case 
\cite{Ben_survey}. The first examples following Thurston's point of view for higher dimensional Lie 
groups were produced by Falbel in \cite{FalJDG}, who constructed and studied examples of representations
of the figure 8 knot group to SU(2,1) (see also \cite{DF8,Falb-Wang}).
A parallel is also to be drawn with higher Teichm\"uller theory in case of surfaces. 
In this note, we focus on the target group $\SL(3,\C)$.

Though simple in spirit, this method of describing representation varieties becomes very involved when the number of tetrahedra 
grows. In fact, the only $\SL(3,\C)$-character variety of a hyperbolic $3$-manifold that has been completely described so 
far with this approach is the one of the figure 8 knot complement \cite{FGKRT}, which admits an ideal triangulation by two tetrahedra 
(see  Section 3.1 of \cite{Thu}). Different methods have been used to describe examples of character varieties. In \cite{HMP}, Heusener, 
Mu\~noz and Porti gave another description of the character variety of the figure 8 group starting directly from the group presentation.
In \cite{MunozPorti}, Munoz and Porti described character varieties for torus knots. We will consider here the example of 
the Whitehead link complement (which can be triangulated by 4 ideal simplices). Denote by $\pi$ its fundamental group. 
Denoting by $[x,y]=xyx^{-1}y^{-1}$ the commutator of $x$ and $y$, a possible presentation for $\pi$ is:
\begin{equation*}
 \la x,y\, \vert\, [x,y][x,y^{-1}][x^{-1},y^{-1}][x^{-1},y]\ra.
\end{equation*}
Let $\chi_3(\pi)$ be the corresponding $\SL(3,\C)$-character variety, that is the GIT quotient
$$ \chi_3(\pi)={\rm Hom}(\pi,\SL(3,\C))//\SL(3,\C).$$
The full computation of $\chi_3(\pi)$ is not achieved as of today and seems a difficult task.
Our goal here is to describe an algebraic component of $\chi_3(\pi)$ that
contains many examples of geometrically meaningful representations.

Our motivation comes from  the study of the so-called 
{\it spherical CR structures} on hyperbolic 3-manifolds. These structures are examples of $(G,X)$-structures 
where $X$ is $\mathbf{S}^3$ and $G$ is PU(2,1). The holonomy of such a structure is thus a representation of $\pi_1(M)$ in 
PU(2,1). This motivates the study of representations of $\pi$ in $\textrm{SU}(2,1)$ which is a real form of SL(3,$\C$). Recall that PU$(2,1)$ is the group of holomorphic isometries of 
the complex hyperbolic plane $\HdC$ and SU$(2,1)$ is a triple cover of it. The sphere $\mathbf{S}^3$ is the boundary at infinity of $\HdC$. In particular, 
spherical CR structures arise naturally on the boundary
at infinity of quotients of $\HdC$ with non-empty region of discontinuity. 
Spherical CR structures can also be thought of as examples of {\it projective flag structures} on 3-manifolds on $M$, of which 
holonomies  are representations of $\pi_1(M)$ to PSL(3,$\C$).

Striking examples of spherical CR structures have been produced by R. Schwartz in \cite{S,Schbook} about fifteen years ago. There, 
Schwartz described what is now called a {\it spherical CR uniformisation} of the Whitehead link complement, that is a spherical 
CR structure with the additional property that the holonomy representation has non-empty region of dicontinuity with 
quotient homeomorphic to the Whitehead link complement (see \cite{Derrep} for a precise definition). Since then, Deraux 
and Falbel \cite{DF8} produced a spherical CR uniformisation
 of the complement of the figure eight knot, Deraux \cite{DerF8} and Acosta \cite{Aco} deformed this uniformisation, Deraux \cite{Derrep} described a uniformisation of the manifold $m_{009}$, and Parker-Will \cite{ParkWi2} described 
another uniformisation of the Whitehead link complement, different from Schwartz's one.

 Our goal here is to provide a common frame for all these examples. The crucial remark is the following: the group $\pi$ 
has $\pi ' = \Z_3 * \Z_3$ as a quotient.  Moreover, it has been observed that all the representations alluded to 
above have two features in common: 
\begin{enumerate}
\item their sources are quotients of $\pi$ (hence these representations can be seen as representations of $\pi$) 
\item these representations factor through $\pi'$. 
\end{enumerate}

The second item can be rephrased by saying that the images of all these representations
are subgroups of PU$(2,1)$ generated by a pair of regular order three elements (see the 
introduction of \cite{ParkWi2} for a list of groups generated by two regular order three elements that are known to be 
discrete and provide examples of spherical CR structures on various link complements).

We are going to prove that a component of the SL(3,$\C$)-character variety of $\pi$ is formed
by representations that factor through $\pi'$. Moreover this component contains all the representations mentionned 
above. Let us describe this component more precisely. Define $X_0$ as the subset of the character variety of $\pi'$ corresponding to representations 
generated by two regular order three elements of $\SL(3,\C)$ (recall that an order three element in $\SL(3,\C)$ is regular 
if and only if its trace is $0$). Our main result is the following.

\begin{theo}\label{thm:main}
The character variety $\chi_3(\pi)$ contains $X_0$ as an algebraic component of dimension $4$. In particular, all 
representation classes in this component are unfaithful.
\end{theo}

The unfaithfulness part of Theorem \ref{thm:main} is straightforward once noted that all representation in $X_0$ factor 
through $\pi'$.

The $\SL(2,\C)$-character variety $\chi_2(\pi)$ for $\pi$ has been computed in 
\cite{BodenCurtis}. It has two components: one contains the characters of every irreducible representations ; the other 
component contains the characters of every reducible representations. Using the irreducible representation 
$\SL(2,\C) \to \SL(3,\C)$, we get a natural map $\chi_2(\pi) \to \chi_3(\pi)$. Our component $X_0$ of $\chi_3(\pi)$ is new, 
in the sense that it does not contain the image of any component of $\chi_2(\pi)$. Indeed, a generic point in either 
component of $\chi_2(\pi)$, as described in \cite{BodenCurtis}, does not correspond to a representation of $\pi'$.

It should be noted that the group $\pi'$ is the fundamental group of a compact exceptional Dehn filling of the 
Whitehead link complement, as we will see later on. The situation we describe is therefore very
similar to the one of the SL(3,$\C$) character variety of the figure $8$-knot group 
(see \cite[Proposition 10.3]{HMP} or \cite[Section 5.3]{FGKRT}). 
There, a $2$-dimensional component of the character variety is formed by 
representations factoring through a quotient of $\pi'$, which can be viewed as 
the fundamental group of a non-hyperbolic Dehn surgery on the figure $8$-knot complement. 
This quotient is isomorphic to an index 2 subgroup of the $(3,3,4)$-triangle group, which is in turn a 
quotient of $\pi'$. As such (see Proposition \ref{pr:quotient-inclusion}), the aforementioned $2$-dimensional 
component for the $8$-knot complement may be seen as a slice of the $4$-dimensional component for the Whitehead 
link complement described in Theorem \ref{thm:main}.

The proof of Theorem \ref{thm:main} has two steps.
\begin{itemize}
\item First we prove that $X_0$ is a closed Zariski subset in $\chi_3(\pi)$
 and has dimension at least $4$ (see Proposition \ref{pr:lower-bound}).
\item Secondly we consider the particular point of $X_0$ associated to
a representation $\rho_0$ and show that the dimension of the Zariski tangent space to 
$\chi_3(\pi)$ at this point is also $4$ (see Proposition 
\ref{pr:upper-bound}). As a consequence, the dimension of the complex algebraic
variety $X_0$ is at most $4$. The representation $\rho_0$ is defined in Section \ref{ss:upper-bound}.  It has 
been analysed geometrically in \cite{ParkWi2} and corresponds to a spherical CR uniformisation of the 
Whitehead link complement.
\end{itemize}

The main technical part in our work is the proof of Proposition 
\ref{pr:upper-bound}. We choose to prove this proposition using a method that
is not specific to the Whitehead link complement, and we believe that it could  be used to 
study further examples. It involves the so-called deformation variety as described in \cite{FGKRT}. 
The latter is an affine algebraic set, which is -- at least around
$[\rho_0]$ -- a ramified covering of the character variety. The purpose of 
shifting to the deformation variety is that it allows effective computations
via decorated representations and triangulations, as in \cite{FGKRT}.

In the last section, we describe an explicit family of pairwise unconjugate
representations of $\pi'$, whose conjugacy classes form a Zariski open 
subset of $X_0$. Our motivation for describing this generic set of 
actual matrices in the component $X_0$ is that this component should provide a nice 
playground for experimentations with the geometric structures we mentionned above. Having explicit 
expressions is very useful for that. This is called, in the works of Culler, Morgan and Shalen \cite{CS,MS}
a \textit{tautological representation} of $X_0$ \cite{CS,MS}. These representations are defined by pairs of regular 
order three matrices $(A,B)$ with no common eigenvector that are parametrised by the traces
of the four products $AB$, $A^{-1}B$, $A^{-1}B^{-1}$ and $AB^{-1}$. These parameters
are natural coordinates on $X_0$, in view of Lawton's description of the character variety
of the rank $2$ free group given in Theorem \ref{theo-lawton} (see \cite{Law}). 

\medskip

This paper is organised as follows. In Section \ref{s:hyp-geom}, we describe
the Whitehead link complement and its fundamental group from the perspective
of (real) hyperbolic geometry. In particular, we gather together classical
information on presentations and parabolic subgroups of $\pi$ that will be
needed further. Section \ref{s:char-var} is devoted to the character variety
 of $\pi$.
We provide basic definitions and facts on these objects. We prove Proposition 
\ref{pr:lower-bound}, state Proposition \ref{pr:upper-bound} and derive 
Theorem \ref{thm:main} from them. In Section \ref{s:defor}, 
we present the deformation variety and prove Proposition \ref{pr:upper-bound}. Eventually, in Section 
\ref{s:param}, we describe an explicit parametrisation of the representations in $X_0$.
The interested reader may also want to use the companion Sage notebook \cite{Notebook} which 
combines the use of SnapPy \cite{SnapPy} and SageMath \cite{sage} to illustrate our method.

\medskip

\noindent {\bf Aknowledgement:} We wish to thank Miguel Acosta, Martin Deraux, Elisha Falbel, Michael Heusener
and John Parker for numerous stimulating conversations. 
We thank Neil Hofmann, Craig Hodgson, Bruno Martelli and 
Carlo Petronio for kindly answering our questions.

\section{Hyperbolic geometry of the Whitehead Link Complement}\label{s:hyp-geom}

The Whitehead link is depicted on Figure \ref{fig-WL}.  We denote by $W$ its complement in $\mathbf S^3$ and by $\pi$ 
the fundamental group of $W$. 

\subsection{The ideal octahedron}\label{ss:octa}
It is a well-known fact that $W$ carries a (unique) complete hyperbolic structure 
(we refer to \cite{Thu} for more details). This structure can be explicitly described by considering an 
ideal regular octahedron in the real hyperbolic $3$-space ( see Figure \ref{fig-WL}), together  with identifications 
of its faces by isometries. We refer the reader to Section 3.3 of \cite{Thu}, or to Example 1 in \cite{Wiel} for details. 
We briefly recall here the description of this structure which is given in \cite{Wiel}. 

\begin{figure}[ht]
\begin{tabular}{lr}
\begin{minipage}[c]{.46\linewidth}
 \scalebox{0.6}{\includegraphics{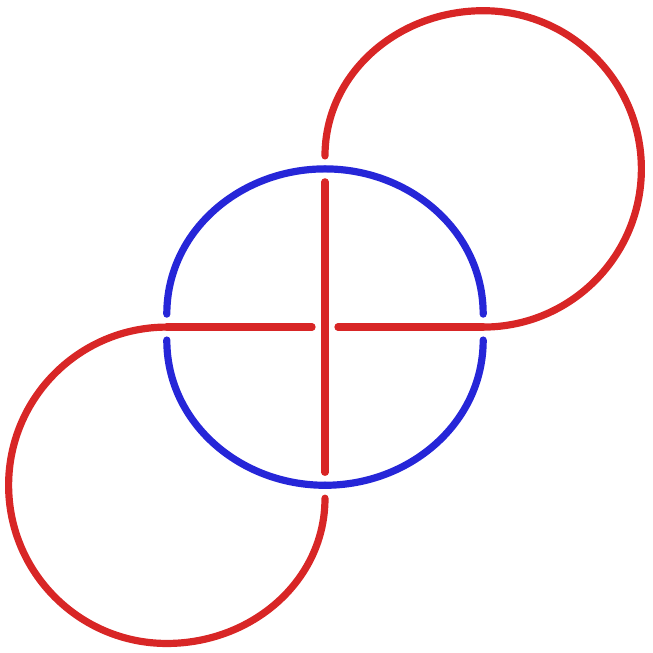}}
\end{minipage}
&\begin{minipage}[c]{.46\linewidth}
 \scalebox{0.4}{\includegraphics{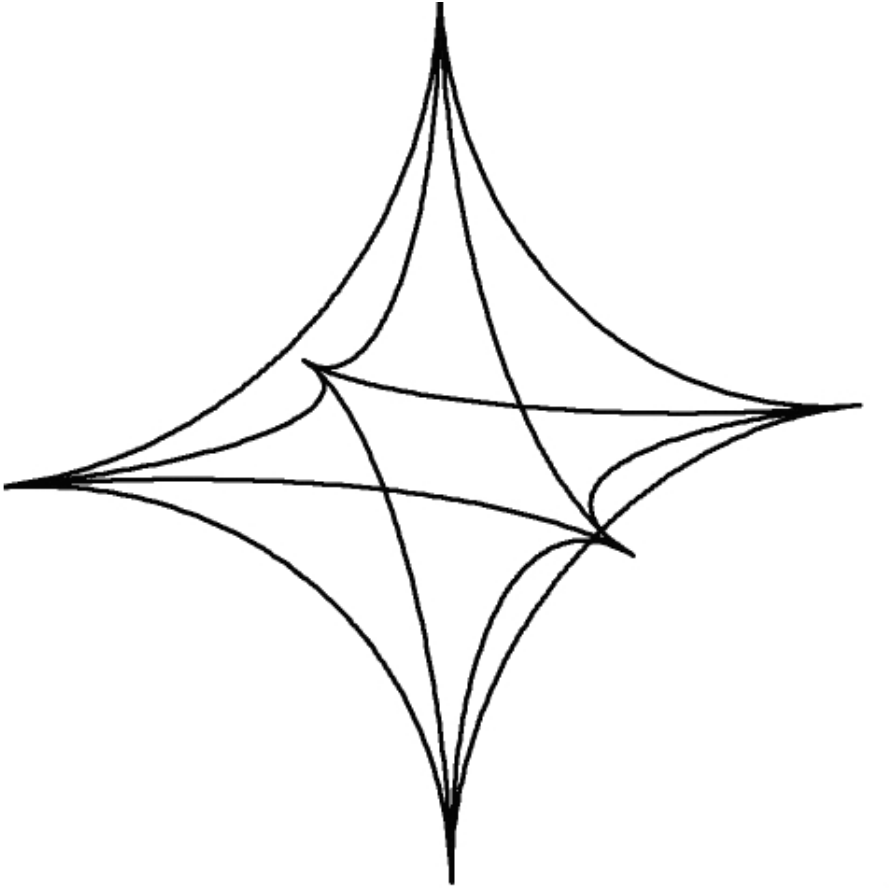}}
\end{minipage}
\end{tabular}
\caption{The Whitehead link, and a hyperbolic regular ideal octahedron. }\label{fig-WL}
\end{figure}
%\begin{figure}[ht]
%\begin{tabular}{lr}
%\begin{minipage}[c]{.46\linewidth}
% \scalebox{0.6}{\includegraphics{Whitehead3.eps}}
%\end{minipage}
%&\begin{minipage}[c]{.46\linewidth}
% \scalebox{0.4}{\includegraphics{octa-hyp-2.eps}}
%\end{minipage}
%\end{tabular}
%\caption{The Whitehead link, and a hyperbolic regular ideal octahedron. }\label{fig-WL}
%\end{figure}

We denote by $\mathcal{O}$ the ideal octahedron of which vertices are given, in the upper half-space model of $\HtR$ by 
$$ \infty,\, 0,\, -1,\,-1+i,\, i,\, \dfrac{-1+i}{2}.$$
A flattened version of this octahedron is pictured in Figure \ref{fig-glue-octahedron}. We denote by $x$, $y$ and $t$ the 
isometries of $\HtR$ associated to the following elements of SL(2,$\C$).
\begin{equation}\label{xyt}
x=\begin{bmatrix} 1 & i \\ {\black 0} & 1\end{bmatrix},\, t=\begin{bmatrix}1 & 2 \\ {\black 0} & 1\end{bmatrix} {\rm and }\,
y=\begin{bmatrix} 1 & {\black 0} \\ -1-i & 1\end{bmatrix}.
\end{equation}

In Wielenberg's article, $x$ corresponds to $u$, $t$ to $t_2$ and $y$ to $w_1$. Let $\Gamma$ be the subgroup of PSL(2,$\C$) 
generated by $x$ and $y$.

Note that $t$ belongs to $\Gamma$ since $t=[y^{-1},x][y,x]x^2$. Now, we equip $\mathcal{O}$ with the face identifications 
described on Figure \ref{fig-glue-octahedron}. This particular choice gives a holonomy representation with image $\Gamma$, 
which is isomorphic to the fundamental group of the Whitehead link complement, and can be seen to have index twelve 
in the Bianchi group PSL(2,$\Z[i]$) (see \cite{Wiel}). The octahedron from Figure \ref{fig-glue-octahedron} is a fundamental 
domain for $\Gamma$, and applying Poincar\'e's polyhedron theorem leads to the following classical presentation for $\pi$.
\begin{equation}\label{present-pi-standard}
 \la x,y\, \vert\, [x,y][x,y^{-1}][x^{-1},y^{-1}][x^{-1},y]\ra.
\end{equation}
We refer the reader to chapter 11 of \cite{Rat} for an exposition of Poincar\'e's polyhedron theorem

\begin{figure}
 \begin{tabular}{lr}
\begin{minipage}[c]{.46\linewidth}
 \scalebox{0.35}{\includegraphics{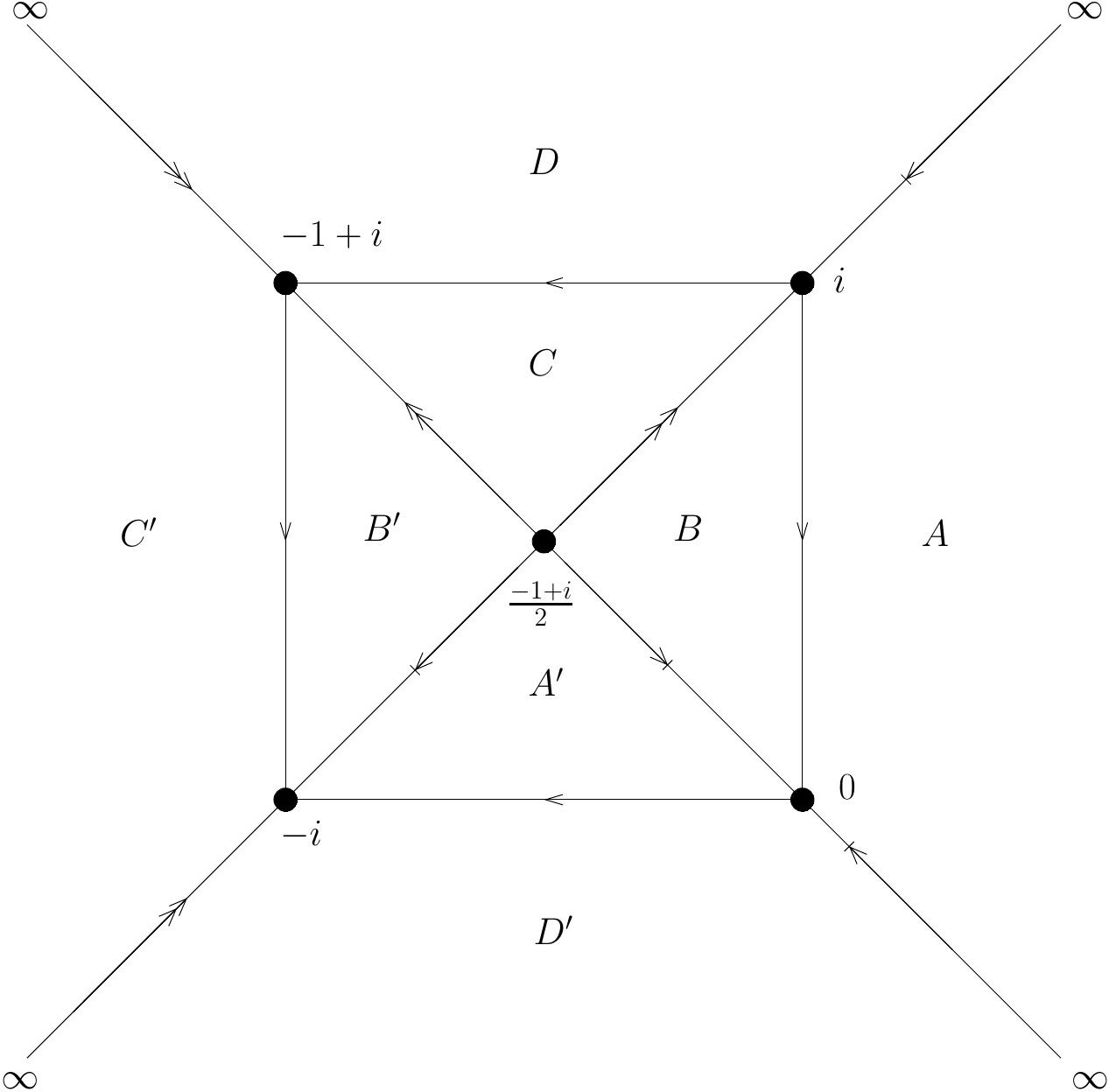}}
\end{minipage}
& \begin{minipage}[c]{.46\linewidth}
\begin{tabular}{ll}
 $\alpha=yx^{-1}$ &: $A \longmapsto A'$\\
&\\
 $\beta=yx^{-1}y^{-1}$ &: $B  \longmapsto B'$\\
&\\
 $\gamma=t^{-1}x y^{-1}$ &: $C \longmapsto C'$\\
&\\
 $\delta = x^{-1}$ &: $D  \longmapsto D'$,
\end{tabular}\\
\end{minipage}
\end{tabular}
\caption{ An octahedron with face identifications }\label{fig-glue-octahedron}
\end{figure}
 
\subsection{Triangulating the ideal octahedron}\label{ss:triangulation}
The presentation of the Whitehead link group that we are going to use is in fact not \eqref{present-pi-standard}, but
\begin{equation}\label{present-pi-SnapPy-2}
 \la a,b \vert [ba^{-3}b^2,a^{-1}b]\ra.
\end{equation}
An isomorphism between the two presentations is given by the changes of generators
\begin{equation}\label{isom-present}
 (x,y)=(ab^{-1},ab^{-1}a) \mbox{ and } (a,b)=(x^{-1}y,x^{-2}y).
\end{equation}
This isomorphism appears in \cite{ParkWi2} (Proposition 3.3), where it has a geometric meaning.
The presentation \eqref{present-pi-SnapPy-2} is the one provided by the software SnapPy \cite{SnapPy}. It is worth 
noting that the computation made by SnapPy is based on an ideal triangulation of the Whitehead link complement, which 
is going to be of use for us, and is as follows. Connect the vertices  with coordinates $\infty$ and $\frac{-1+i}{2}$ 
by an edge. One obtains a decomposition  of the octahedron $\mathcal{O}$ as a union of four ideal tetrahedra, which we 
label as follows (see Figure \ref{fig-glue-octahedron}).

\begin{equation}\label{simplices}
\begin{array}{ll}
\Delta_0 = \Bigl(i,0,\dfrac{-1+i}{2},\infty\Bigr) & \Delta_1 = \Bigl(-1 + i,i, \dfrac{-1+i}{2},\infty\Bigr) \\
&\\
\Delta_2 = \Bigl(-i,-1+i,\dfrac{-1+i}{2},\infty\Bigr)& \Delta_3 =  \Bigl(0,-i,\dfrac{-1+i}{2},\infty\Bigr)
\end{array}
\end{equation}

%\noindent This gives an ideal triangulation of the manifold $W$ which is the one used by the software SnapPy \cite{SnapPy} to
%provide another presentation of $\pi$ given by  
%\begin{equation}\label{present-pi-SnapPy}
% \la a,b \vert aba^{-3}b^{2}a^{-1}b^{-1}a^3b^{-2}\ra.
%\end{equation}
%Since the relator in \eqref{present-pi-SnapPy} satisfies 
%$ aba^{-3}b^{2}a^{-1}b^{-1}a^3b^{-2} = a [ba^{-3}b^2,a^{-1}b] a^{-1}$, this second presentation is actually equivalent to 
%\begin{equation}\label{present-pi-SnapPy-2}
% \la a,b \vert [ba^{-3}b^2,a^{-1}b]\ra.
%\end{equation}
%which is used in \cite{ParkWi2}.

%It may be verified, for instance using the software Regina (CITE SOMETHING ?), that the triangulation used by SnapPy to compute 
%\eqref{present-pi-SnapPy} is  obtained by cutting the octahedron into four simplices by connecting by an edge the top and bottom 
%vertices on Figure \ref{fig-WL}. Since the relator in \eqref{present-pi-SnapPy} satisfies 
%$ aba^{-3}b^{2}a^{-1}b^{-1}a^3b^{-2} = a [ba^{-3}b^2,a^{-1}b] a^{-1}$, this second presentation is actually equivalent to 
%\begin{equation}\label{present-pi-SnapPy-2}
% \la a,b \vert [ba^{-3}b^2,a^{-1}b]\ra.
%\end{equation}
%\noindent An isomorphism between \eqref{present-pi-standard} and \eqref{present-pi-SnapPy} is given by the changes of 
%generators 
%\begin{equation}\label{isom-present}
 %(x,y)=(ab^{-1},ab^{-1}a) \mbox{ and } (a,b)=(x^{-1}y,x^{-2}y).
%\end{equation}

%This isomorphism appears in \cite{ParkWi2} (Proposition 3.3), where it has a geometrical meaning.

The geometric representation of $\pi$ is defined (up to PSL(2,$\C$)-conjugation) as the one corresponding to the unique finite 
volume hyperbolic structure on the Whitehead link complement. We denote it by $\rho_{\rm geom}$. It can be described in terms 
of either presentations \eqref{present-pi-standard} or \eqref{present-pi-SnapPy-2}. We now provide matrices for the 
images of $a$ and $b$ in the presentation \eqref{present-pi-SnapPy-2} that can be easily related to \eqref{present-pi-standard} 
by the isomorphism \eqref{isom-present}. 

\begin{prop}
 The geometric representation of $\pi$ is the morphism $\rho_{\rm geom} : \pi\longrightarrow {\rm PSL}(2,\C)$ defined by  
\begin{equation}\label{def-rho-geom}
 \rho_{\rm geom}(a)= x^{-1}y=\begin{bmatrix}i & -i \\ -1-i & 1\end{bmatrix}\mbox{ and }\rho_{\rm geom}(b)= x^{-2}y=\begin{bmatrix} -1 + 2i & -2i \\ -1-i & 1\end{bmatrix}.
\end{equation}
\end{prop}
\medskip
The geometric representation is thus a discrete and faithful representation of $\pi$ in PSL(2,$\C$) with image $\Gamma$.
%The above matrices for $a$ and $b$ are obtained by identifying respectively $x$ and $y$ to $u$ and $w_1$ in \eqref{isom-present}.
Note that $\rho(b^{-1}a^3b^{-1}a^{-1})=t$. 

\subsection{Stabilizers of the vertices and peripheral curves}
There are two orbits of vertices modulo the identifications in the octahedron $\mathcal{O}$: the one of 
$\infty$ and the one of $0$. It is a simple exercise using the face identifications to verify that the stabilizers 
of these two points are respectively $\Gamma_\infty=\la x, t\ra$, and $\Gamma_0=\la y,xy^{-1}xt^{-1}x^{-1}y^{-1}x\ra$.
The second generator for $\Gamma_0$ is  the projective transformation associated to
$$\begin{bmatrix}
   1 & {\black 0} & \\ 2-2i & 1
  \end{bmatrix}.$$
We express now these stabilizers in terms of $a$ and $b$.
\begin{prop}\label{pr:stabil}
 The stabilizers of $0$ and $\infty$ in $\Gamma$ are the images of the two subgroups of $\pi$ 
respectively given by $\la ab^{-1}a,s_0\ra$ and $\la ab^{-1},s_\infty\ra$ by $\rho_{geom}$, where
$s_0=[a,b^{-1}]a^{-1}b^2a^{-3}[b,a]$ and $s_\infty=b^{-1}a^3b^{-1}a^{-1}$.
\end{prop}

\begin{rem}\label{rem:long-mer}
 SnapPy provides the following generators for the first homology group of the boundary tori of the Whitehead link complement 
using the presentation \eqref{present-pi-SnapPy-2} ($m_i$ stands for meridian, and $l_i$ for longitude):
\begin{equation}
m_1=a^{-2}b,\, l_1=a^{-2}bab^{-2}ab,\, m_2=b^{-1}a,\, l_2= b^{-1}ab^{-1}aba^{-3}ba.
\end{equation}
By a direct computation, one verifies that 
\begin{eqnarray*}
 \rho_{\rm geom}(am_1a^{-1})=y & \rho_{\rm geom}(al_1a^{-1})=y^{-2}s_0 \\
\rho_{\rm geom}(am_2a^{-1})=x & \rho_{\rm geom}(al_2a^{-1})=t^{-1}x^2
\end{eqnarray*}
We see therefore that $m_1$ and $l_1$ correspond to the cusp of $W$  associated to (the orbit of) $0$, and that $m_2$ and $l_2$ 
correspond to the one associated to (the orbit of) $\infty$.
\end{rem}

\section{The SL(3,$\C$)-character variety}\label{s:char-var}

\subsection{Generalities}

\begin{defi}
Let $G$ be a finitely generated group. The {\it representation variety} of $G$ in SL(3,$\C$) is 
$${\rm Hom}(G,{\rm SL}(3,\C)).$$ 
Its GIT quotient
$${\rm Hom}(G,{\rm SL}(3,\C))//{\rm SL}(3,\C),$$

where the action is by conjugation of representations, is called the SL(3,$\C$)-{\it character variety}, denoted 
by $\chi_3(G)$.
\end{defi}

We refer the reader to \cite{Si2,Heu-survey} for classical definitions 
about representation and character varieties and associated objects.
A remark is important for our purposes: if $G'$ is a quotient of $G$, 
then there is a natural map:
$$\Hom(G',{\rm SL}(3,\C)) \hookrightarrow \Hom(G,{\rm SL}(3,\C)).$$
Indeed, a representation $\rho' : G' \to {\rm SL}(3,\C)$ is naturally 
promoted to a representation 
$$\rho : G \twoheadrightarrow G' \xrightarrow{\rho'} {\rm SL}(3,\C).$$
As the projection $G \twoheadrightarrow G'$ is surjective, this map
is injective and moreover two representation $\rho'$ and $\bar \rho'$
in $\Hom(G',{\rm SL}(3,\C))$ are conjugate if and only if their associated $\rho$ and
$\bar \rho$ in $\Hom(G,{\rm SL}(3,\C))$ are conjugate. From Proposition 1.7 in \cite{LubMag}, 
the map on the level of representation variety is a closed immersion covariant for the action 
of PGL(3,$\C$) by conjugation. From the discussion in page 15 of \cite{LubMag}, this map 
induces a closed immersion on the level of character varieties. We obtain therefore
\begin{prop}\label{pr:quotient-inclusion}
If $G'$ is a quotient of $G$, then $\chi_3(G') \subset \chi_3(G)$.
\end{prop}

\subsection{A quotient of $\pi$}

We denote by $\pi'$ the quotient of $\pi$ defined by the extra relations $a^3=b^3=1$. More precisely, 
a presentation for $\pi'$ is given by $\pi' = \la a,b \vert a^3, b^3, [ba^{-3}b^2,a^{-1}b]\ra$. 
Clearly, the group $\pi'$ is isomorphic to $\Z_3\star \Z_3$: the last relation is a 
consequence of the first two.  Moreover, SnapPy indicates that $\pi'$ is the fundamental group of a 
double Dehn surgery on the Whitehead link. We make this statement precise as follows.
\begin{prop}\label{prop-dehn}
The group $\pi'\simeq \Z_3\star \Z_3$ is isomorphic to the quotient of $\pi$ 
defined by the two relations $m_1^3l_1^{-1}$ and $m_2^3l_2^{-1}$.
\end{prop}
The proof of Proposition \ref{prop-dehn} is a direct verification from the 
definition of $l_i$ and $m_i$ given in Remark \ref{rem:long-mer}: the two conditions
$m_1^3l_1^{-1}=m_2^3l_2^{-1}=1$ imply that $a^3=b^3=1$. In  terms of Dehn surgery, $\pi'$ is the fundamental 
group of the double Dehn surgery of slopes $(-3,-3)$ on the Whitehead link. This double Dehn surgery is not 
hyperbolic: this may be verified using SnapPy (see the companion Sage notebook to this paper 
\cite{Notebook}). More precisely, it can be seen to be the connected sum of two 
lens spaces (see \cite[Table 2]{MP}). 

%This 
%follows from the classification of exceptional surgeries on the Magic manifold 
%given by Martelli 
%and Petronio \cite[Table 2]{MP} (the Magic manifold is the complement of the 
%chain link with $3$ components, see also \cite[Section 6.8]{Thu}). Our Dehn 
%filling of the Whitehead link complement corresponds to the one with 
%parameters  $\frac{p}{q}=\frac{r}{s} = -2$ and $\frac{t}{u} = 1$ of the Magic 
%manifold (following the notation of Martelli-Petronio).

\subsection{A lower bound for $\dim(X_0)$}

We are now going to describe the SL(3,$\C$)-character variety of $\pi'$. To this end, we use Lawton's theorem on the 
SL(3,$\C$)-character variety of the rank two free group $F_2$ \cite{Law}.

\begin{theo}[Lawton \cite{Law}]\label{theo-lawton}
 The map $\psi$ defined by
$$\begin{array}{ccl}
 \SL(3,\C)\times \SL(3,\C)& \rightarrow  &\C^8\\
(A,B)&\mapsto &(\tr A,\tr B, \tr AB,\tr A^{-1}B,\tr A^{-1},\tr B^{-1},\tr A^{-1}B^{-1},\tr AB^{-1})
\end{array}
$$
is onto $\C^8$ and descends to a (double) branched cover $\underline{\psi} : \chi_3(F_2)\longrightarrow\C^8$.
\end{theo}
The theorem in Lawton's work is more precise and gives an explicit polynomial
 in $9$ variables defining $\chi_3(F_2)$ as a hypersurface in $\C^9$ covering 
$\C^8$. Namely, the above double cover corresponds to the fact that the traces of the nine words
$A$, $B$, $AB$, $A^{-1}B$, $A^{-1}$, $B^{-1}$, $A^{-1}B^{-1}$, $AB^{-1}$ and $[A,B]$ satisfy a relation of the form
\begin{equation}(\tr [A,B])^2-S \cdot \tr[A,B] +P =0,\label{trace-equation}\end{equation}
where $S$ and $P$ are polynomials in the traces of the first eight above words. In other words, once the traces of 
$A$, $B$, $AB$, $A^{-1}B$, $A^{-1}$, $B^{-1}$, $A^{-1}B^{-1}$, $AB^{-1}$ are fixed, the trace of $[A,B]$ is determined 
up to the choice of a root of \eqref{trace-equation}. We provide the precise values of $S$ and $P$ in the 
last section of the Sage notebook \cite{Notebook}, they may also be found in Lawton's \cite{Law}, or in \cite{Wi8}.

We can now give an alternate definition of the set $X_0$ considered in the 
introduction: 
\begin{defi}
Let $X_0\subset\chi_3(F_2)$ be the inverse image by $\underline{\psi}$ of the subspace $V$ of $\C^8$ given by 
$$V=\lbrace(0,0,z_1,z_2,0,0,z_3,z_4),z_i\in\C\rbrace.$$
\end{defi}

\noindent By Proposition \ref{pr:quotient-inclusion}, the sequence of quotients $F_2 \twoheadrightarrow \pi \twoheadrightarrow \pi'$ gives rise to a sequence of inclusions:
$$\chi_3(\pi')\subset \chi_3(\pi) \subset \chi_3(F_2).$$
With these inclusions in mind, we see that $X_0$ is actually included in $\chi_3(\pi')$.  We can even be more specific:
\begin{prop}\label{pr:lower-bound}
 The set $X_0$ is an irreducible Zariski closed subset of $\chi_3(\pi')$. Its dimension is at least $4$.
\end{prop}
\begin{proof}
First, $X_0$ is included in $\chi_3(\pi')$. Indeed, the condition $\psi(A,B)\in V$ rewrites 
$$\tr A=\tr A^{-1} = \tr B =\tr B^{-1}=0$$
This implies that both $A$ and $B$ have order three. Indeed, the characteristic polynomial of a 
matrix $M\in\mbox{SL(3,$\C$)}$ is equal to $X^3-\tr M X^2+\tr M^{-1} X -1$ and thus if $\tr M=\tr M^{-1}=0$ we have $M^3=Id$.

By construction, $X_0$ is Zariski closed. Its irreducibility is not hard to verify using
the explicit form of the branched $2$-cover of Theorem \ref{theo-lawton} given
in \cite{Law}. For example, using the parametrisation given in Section \ref{s:param}, it is easily
seen that the double cover is indeed a branched one and not the union of two distinct sheets: the 
discriminant of the quadratic equation defining the double cover is not a square. For a direct and 
precise proof, see also \cite[Section 5.1]{Aco2}. The dimension estimate follows from the fact it is the 
pull-back of $\C^4$ by $\underline{\psi}$.
\end{proof}

\begin{rem}
 In fact $X_0$ is a component of $\chi_3(\pi')$ and the only one of positive dimension, as proved by Acosta in 
\cite{Aco2}.  Moreover it contains every irreducible representation of $\pi'$.
 The character variety $\chi_3(\pi')$, with a focus on real points,
  has been studied in details in \cite{Aco2}.
\end{rem}

\subsection{An upper bound for $\dim(X_0)$}\label{ss:upper-bound}

We give in this section an upper bound on the dimension of the component of
$\chi_3(\pi)$ containing $X_0$ by looking at a specific point to determine 
the Zariski tangent space. To this end, we consider the following two elements 
of $\SL(3,\C)$:
%$$
%S= \begin{pmatrix} 1 & \frac{\sqrt{3} - i\sqrt{5}}{2} & -1\\
%\frac{-\sqrt{3} - i\sqrt{5}}{2} & -1 & 0\\
%-1 & {\black 0} & {\black 0} \end{pmatrix}
%\quad \textrm{,  }\quad
%T= \begin{pmatrix} 1 & -\frac{\sqrt{3} + i\sqrt{5}}{2} & -1\\
%\frac{\sqrt{3} - i\sqrt{5}}{2} & -1 & 0\\
%-1 & {\black 0} & 0\\
% \end{pmatrix}, {\rm where} q=\dfrac{\sqrt{3}+i\sqrt{5}}{2}
%$$

\begin{equation}\label{eq:defST}
S= \begin{pmatrix} 1 & \overline{q} & -1\\
-q & -1 & 0\\
-1 & {\black 0} & {\black 0} \end{pmatrix}
\quad \textrm{,  }\quad
T= \begin{pmatrix} 1 & -q & -1\\
\overline{q} & -1 & 0\\
-1 & {\black 0} & 0\\
 \end{pmatrix},{\rm where }~ q=\dfrac{\sqrt{3}+i\sqrt{5}}{2}.
\end{equation}

The matrices $S$ and $T$ have order three and satisfy 
$\tr(S)=\tr(S^{-1})=\tr(T)=\tr(T^{-1})=0.$
We define a point $[\rho_0]$ in the character variety of $\pi$ by setting:
\begin{equation}\label{rho0}
 \rho_0(a)=S,\, \rho_0(b)=T.
\end{equation}
By construction, $[\rho_0]$ belongs to $X_0$. The key step in the proof of Theorem \ref{thm:main} is 
the following

\begin{prop}\label{pr:upper-bound}
The Zariski tangent space to $\chi_3(\pi)$ at $[\rho_0]$ has dimension $4$.
\end{prop}

We postpone the proof of Proposition \ref{pr:upper-bound} to the next section and proceed with the proof of Theorem \ref{thm:main}. 
Recall that we need to prove that $X_0$ is an algebraic component of $\chi_3(\pi)$ containing $[\rho_0]$.

\begin{proof}[Proof of Theorem \ref{thm:main}]
Let $X$ be an algebraic component of $\chi_3(\pi)$ containing $X_0$. The class $[\rho_0]$ belongs to $X$. By Proposition \ref{pr:upper-bound}, the dimension of $X$ is at most $4$: 
it is bounded above by the dimension of any Zariski tangent space. But, by 
Proposition \ref{pr:lower-bound}, $X_0\subset X$ is a Zariski closed subset of dimension $4$. Hence, $X=X_0$ and the theorem is proved.
\end{proof}

\section{Decorated representations and the deformation variety}\label{s:defor}
We are going to compute the dimension of the Zariski tangent space of $\chi_3(\pi)$ at $[\rho_0]$, in order to prove 
Proposition \ref{pr:upper-bound}. To this end, we will use a variation of the character variety -- called 
{\it decorated character variety} -- and a specific set of coordinates on it -- the {\it deformation variety}. 
The deformation variety is well-adapted to explicit computations. The equations defining
this variety may be reconstructed using SnapPy's command \texttt{gluing\_equations\_pgl} \cite{SnapPy}. 

The tools hereafter presented are suitable for character varieties with target  group the 
quotient $\PGL(3,\C)$ rather than $\SL(3,\C)$. This will not be
a problem, as we use these tools for computing local dimension around a point
which belongs to both character varieties. Indeed, if $\rho$ is a
representation of $\pi$ in $\SL(3,\C)$, the local dimensions at $[\rho]$ 
of the character varieties for $\SL(3,\C)$ and $\PGL(3,\C)$ are the same. 
As we explain afterwards, around a sufficiently generic representation the decorated representation  
variety is a ramified coveringof the character variety. So they share the same local dimension and 
computations can effectively be done at the level of the deformation variety.

\subsection{Decorated representations}

We first recall basic definitions. More details can be found in \cite{BFGKR}. 
\begin{defi}
 A \textit{flag} of $\PP(\C^3)$ is a pair $([x],[f])$ in  $\PP(\C^3)\times \PP((\C^3)^\vee)$ such that $f(x)=0$. We denote by 
$\mathcal{F}l_3$ the set of flags of $\PP(\C^3)$.
\end{defi}
Geometrically a flag is a pair formed by a point in  $\PP(\C^3)$ and a projective line containing it.

\begin{defi}
Let $\Gamma$ be the fundamental group of a finite volume, cusped hyperbolic manifold $M$, let $\mathcal{P}\subset\HtR$ be the 
set of parabolic fixed points of $\Gamma$ and let $\rho$ be a representation 
$\rho : \Gamma \longrightarrow {\rm PGL}(3,\C)$. A 
{\it decoration} of $\rho$ is a map $\phi : \mathcal{P} \longrightarrow \mathcal{F}l_3$
 which is ($\Gamma,\rho$)-equivariant. 
A pair $(\rho,\phi)$ is called a {\it decorated representation}.
\end{defi}

\begin{defi}
 The {\it decorated representation variety} is 
$${\rm DecHom}(\Gamma) = \lbrace (\rho,\phi), \rho\in{\rm Hom}(\Gamma,{\rm PGL}(3,\C)),\phi \mbox{ is a decoration of }\rho \rbrace.$$
The {\it decorated character variety} is the GIT quotient 
$${\rm Dec}\chi_3={\rm DecHom(\Gamma)}//{\rm PGL}(3,\C) $$
\end{defi}
 The precise links between the different versions of representation and character varieties
are described in detail in the introduction  of \cite{FGKRT}. It should be noted that for a given 
``generic'' representation $\rho : \Gamma \rightarrow \mbox{PGL(3,$\C$)}$, there exists only a finite number of 
possible decorations. Let us explain what we mean here by generic. The set of elements in SL(3,$\C$) that preserve 
only a finite number of flags is Zariski open. We call these elements generic. Now a representation will be called generic 
whenever the image of any peripheral subgroup contains at least one generic element. In this case, if $p\in\mathcal{P}$ is 
the (global) fixed point of a peripheral subgroup $\Gamma_p$, it should be mapped by $\phi$ to a flag that is invariant for 
$\rho(\Gamma_p)$. By genericity there is only a finite number of possible such flags.  By equivariance, the map $\phi$ 
is completely determined by its values on a choice of representatives of the orbits of $\Gamma$ on $\mathcal{P}$. As there 
is a finite number of cusps, the number of possible $\phi$ for a given $\rho$ is finite. Non-generic representation classes 
form a Zariski closed subset of the character variety, which may contain some components.

For our purposes, we chose the point $[\rho_0]$ in $\chi_3(\pi)$ ($\rho_0$ is described in Section \ref{ss:upper-bound}).

\begin{prop}\label{pr:unique-deco}
 The representation $\rho_0$ admits a unique decoration.
\end{prop}

\begin{proof}
 The (hyperbolic) Whitehead link complement has two cusps, which are represented by the stabilizers of $0$ and $\infty$ 
(see Proposition \ref{pr:stabil}). Therefore, the equivariance property implies that a decoration $\phi$ of $\rho_0$ 
is completely determined by the images $\phi(0)$ and $\phi(\infty)$. The images by $\rho_0$ of the stabilizers of $0$ and 
$\infty$ are respectively the cyclic groups $\la ST^{-1}S\ra$ and $\la ST^{-1}\ra$. Indeed, the images of 
the stabilizers of $0$ and $\infty$ by $\rho_0$ are respectively  $\la\rho_0(ab^{-1}a),\rho_0(s_0)\ra$ and 
$\la\rho_0(ab^{-1}),\rho_0(s_\infty)\ra$ (this follows directly from Proposition \ref{pr:stabil}). The images by $\rho_0$ 
of $s_0$ and $s_\infty$ are $ST^{-1}S$  and $TS^{-1}$ : this is a direct verification using $S^3=T^3=1$.

Now, the two maps $ST^{-1}$ and $ST^{-1}S$ are regular unipotent : this means that they are unipotent and that the 
eigenspace for the eigenvalue $1$ is a line. Thus each of them has only one invariant flag. Therefore $\rho_0$ can 
only be decorated in one way : $\phi$ must map $0$ to the invariant flag of $ST^{-1}S$ and $\infty$ to the one of $ST^{-1}$.
\end{proof}

The invariant flags of $ST^{-1}$ and $ST^{-1}S$ (as well as those of various elements in the group) are made explicit in Table 
\ref{flags}. A consequence of Proposition \ref{pr:unique-deco} is that around $[\rho_0]$, the decorated character variety 
is a finite ramified cover of the character variety (see also \cite{Guilloux-discretudePGLn}). As a 
consequence, the local dimension around $[\rho_0]$ can be equivalently computed at the level
of $\chi_3(\pi)$ or of the decorated character variety. 

\subsection{Using a triangulation: the deformation variety\label{ss:defor}}
A configuration of ordered points in a projective space $\PP(V)$ is said to be in general position when they are all distinct 
and no three points are contained in the same line. This notion applies to configurations of projective lines by duality.
A configuration of $n$ flags $(([x_1],[f_1],\cdots,([x_n],[f_n])))$ is in {\it general position} whenever the $n$ points 
$([x_i])_{i=1}^n$ are in general position and the forms $[f_j]$ satisfy $f_j(x_i)\neq 0$ when $i\neq j$.
\begin{defi}
 We call {\it tetrahedon of flags} in $\mathbb P(\C^3)$ any 4-tuple of flags in general position.
\end{defi}
We briefly recall now the definitions of the main projective invariants we are going to use as well as the relations among them.
We refer the reader to \cite{BFG} for more details. Let $T = (F_1, F_2, F_3, F_4)$ be a tetrahedron of flags.
\begin{enumerate}
 \item {\bf Triple ratio. }Let $\bigl(ijk\bigl)$ be a face of T (oriented 
  accordingly to the orientation of the tetrahedron) of flags in general 
  position. Its {\it triple ratio} is the quantity
 \begin{equation}\label{def-triple}
  z_{ijk}=\dfrac{f_i(x_j)f_j(x_k)f_k(x_i)}{f_i(x_k)f_j(x_i)f_k(x_j)}.
 \end{equation}
\item {\bf Cross-ratio. }  Whenever four points $(a,b,c,d)$ lie on a projective line, we denote by $[a,b,c,d]\in\C$ their 
cross-ratio. For each oriented edge $(ij)$ of $T$, we define $k$ and $l$ in such a way that the permutation 
$(1,2,3,4)\longmapsto (i,j,k,l)$ is even. Now, viewing the set of lines through $[x_i]$ as a projective line, we associate 
to $(ij)$ the cross-ratio 
\begin{equation}\label{def-zij}
 z_{ij}=\bigl[\ker(f_i),(x_ix_j),(x_ix_k),(x_ix_l)\bigr],
\end{equation}
where $(xy)$ denotes the projective line through $[x]$ and $[y]$.
\end{enumerate}

These invariants can be thought of as decorating tetrahedra as shown in Figure \ref{fig:coord} : to each face is associated a 
triple ratio, and to each edge are associated a pair of cross-ratios. Namely, the two cross ratios $z_{ij}$ and $z_{ji}$ 
are associated to the edge $(ij)$.

\begin{figure}[ht]      
\begin{center}
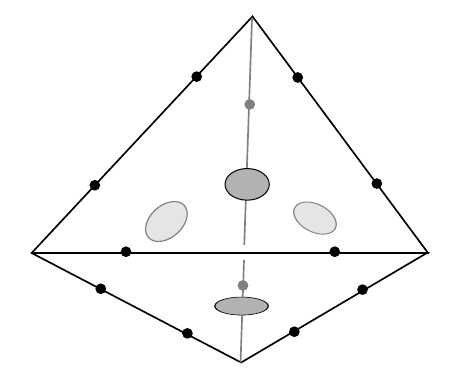
\caption{The $z$-coordinates for a tetrahedron.} \label{fig:coord}
\end{center}
\end{figure}

The above projective invariants are linked by the following internal relations:
\begin{equation*}
z_{ijk}=\dfrac{1}{z_{ikj}},\, z_{ijk}=-z_{il}z_{jl}z_{kl}
\end{equation*}
\begin{equation}
\tag{IR} z_{ik}=\dfrac{1}{1-z_{ij}},\,z_{ij}z_{ik}z_{il}=-1.\label{IR}
\end{equation}
In particular, the triple ratio can be expressed purely in terms of cross ratios.  These invariants 
can be used to parametrise the set of tetrahedra of flags :
\begin{prop}[Proposition 2.10 of \cite{BFG}]
A tetrahedron of flags is uniquely determined up to the action of $\PGL(3,\C)$ by the $4$-tuple $(z_{12},z_{21},z_{34},z_{43})$ in 
$(\C\setminus\{0,1\})^4$. 
\end{prop}
Remark that $0$ and $1$ are forbidden values (as is $\infty$) because we assume the flags
to be in general position: hence every cross-ratio is the cross-ratio of four different points.

\medskip

{Let  now $M$ be an ideally triangulated cusped hyperbolic $3$-manifold. Denote by $\nu$ the number of tetrahedra and by 
$(\Delta_\mu)_{\mu=1}^\nu$ the family of tetrahedra triangulating $M$. We construct a decorated representation of 
$\Gamma = \pi_1(M)$ by turning each tetrahedron into a tetrahedron of 
flags and compute the (decorated) monodromy of their gluing. We only need to ensure 
that we may glue the tetrahedra together in a consistent way : 
\begin{enumerate}
\item whenever two tetrahedra $\Delta$ and $\Delta'$ are glued together along faces $T\subset\Delta$ and $T'\subset\Delta'$, 
$T$ and $T'$ should  have the same shape, that is the same triple ratio up to inversion,
\item around each edge of the triangulation, the monodromy should be the identity.
\end{enumerate}
These gluing conditions are described in details in \cite[Section 2.3]{BFGKR}. They give an 
equation for each face of the triangulation and two for each edge, which are respectively called
the {\it face equations} and the {\it edge equations}.  Together, they are called the 
\emph{gluing equations} and denoted by $(GE)$.

\begin{defi}\label{df:defor}
 The deformation variety of $M$, denoted $\mathrm{Defor}_3(M)$, is the subset of
 $\C^{12\nu}$ given by the $(z_{ij}(\Delta_\mu))_{0\leq i\neq j \leq 3,\, 1\leq \mu\leq \nu}$ 
satisfying the internal relation \eqref{IR} for each tetrahedron $T_\mu$ together with
 the gluing equations (GE).
\end{defi}

%We define the deformation variety $\mathrm{Defor}_3(M)$ to be the subset of 
%$(z_{ij}(T_\mu))_{0\leq i\neq j \leq 3,\, 1\leq \mu\leq \nu}$ in $\C^{12\nu}$
%verifying the internal relation \eqref{IR} for each tetrahedron $T_\mu$ together with
% the face and edge equations.
%The gluing equations for the Whitehead link complement is explicitely described 
%in the following subsection.
\begin{table}[h]
\begin{tabular}{c|c}
Face equations & Edge equations\\
\hline
{\footnotesize $ z_{41}(\Delta_0) z_{31}(\Delta_0) z_{21}(\Delta_0) z_{41}(\Delta_1) z_{31}(\Delta_1) z_{21}(\Delta_1) = 1 $} &
	{\footnotesize $ z_{43}(\Delta_0) z_{34}(\Delta_1) z_{34}(\Delta_2) z_{34}(\Delta_3) = 1 $}  \\ 
{\footnotesize $ z_{42}(\Delta_0) z_{32}(\Delta_0) z_{12}(\Delta_0) z_{42}(\Delta_2) z_{32}(\Delta_2) z_{12}(\Delta_2) = 1 $} & 
	{\footnotesize $ z_{34}(\Delta_0) z_{43}(\Delta_1) z_{43}(\Delta_2) z_{43}(\Delta_3) = 1 $} \\ 
{\footnotesize $ z_{41}(\Delta_2) z_{31}(\Delta_2) z_{21}(\Delta_2) z_{42}(\Delta_3) z_{32}(\Delta_3) z_{12}(\Delta_3) = 1 $} & 
	{\footnotesize $ z_{21}(\Delta_0) z_{12}(\Delta_1) z_{21}(\Delta_2) z_{21}(\Delta_3) = 1 $} \\ 
{\footnotesize $ z_{34}(\Delta_2) z_{24}(\Delta_2) z_{14}(\Delta_2) z_{43}(\Delta_3) z_{23}(\Delta_3) z_{13}(\Delta_3) = 1 $} &
	{\footnotesize $ z_{12}(\Delta_0) z_{21}(\Delta_1) z_{12}(\Delta_2) z_{12}(\Delta_3) = 1 $} \\ 
{\footnotesize $ z_{43}(\Delta_0) z_{23}(\Delta_0) z_{13}(\Delta_0) z_{34}(\Delta_3) z_{24}(\Delta_3) z_{14}(\Delta_3) = 1 $} & 
	{\footnotesize $ z_{24}(\Delta_0) z_{23}(\Delta_0) z_{13}(\Delta_0) z_{24}(\Delta_1) z_{23}(\Delta_1) z_{13}(\Delta_1) z_{14}(\Delta_2) z_{23}(\Delta_3) = 1 $} \\ 
{\footnotesize $ z_{34}(\Delta_0) z_{24}(\Delta_0) z_{14}(\Delta_0) z_{34}(\Delta_1) z_{24}(\Delta_1) z_{14}(\Delta_1) = 1 $} & 
	{\footnotesize $ z_{42}(\Delta_0) z_{31}(\Delta_0) z_{32}(\Delta_0) z_{42}(\Delta_1) z_{31}(\Delta_1) z_{32}(\Delta_1) z_{41}(\Delta_2) z_{32}(\Delta_3) = 1 $} \\ 
{\footnotesize $ z_{42}(\Delta_1) z_{32}(\Delta_1) z_{12}(\Delta_1) z_{41}(\Delta_3) z_{31}(\Delta_3) z_{21}(\Delta_3) = 1 $} &
	{\footnotesize $ z_{41}(\Delta_0) z_{41}(\Delta_1) z_{42}(\Delta_2) z_{31}(\Delta_2) z_{32}(\Delta_2)z_{42}(\Delta_3) z_{41}(\Delta_3) z_{31}(\Delta_3) = 1 $} \\ 
{\footnotesize $ z_{43}(\Delta_1) z_{23}(\Delta_1) z_{13}(\Delta_1) z_{43}(\Delta_2) z_{23}(\Delta_2) z_{13}(\Delta_2) = 1 $} &
	{\footnotesize $ z_{14}(\Delta_0) z_{14}(\Delta_1) z_{24}(\Delta_2) z_{23}(\Delta_2) z_{13}(\Delta_2) z_{24}(\Delta_3) z_{13}(\Delta_3) z_{14}(\Delta_3) = 1 $} \\ 

\end{tabular}
\caption{Gluing equations (GE) for the Whitehead Link Complement}\label{GE}
\end{table}
In the case of the Whitehead Link Complement, the gluing equations are the $16$ monomial equations 
displayed in Table \ref{GE}. Hence, the \emph{deformation variety} of the Whitehead Link Complement 
is the affine algebraic subset of $\C^{48}$ defined by the $32$ internal relations \eqref{IR} and $16$ gluing 
equations of Table \ref{GE}. The holonomy map, as defined in \cite{BFG, GGZ}, is a well-defined 
map from the deformation variety to the character variety 
$\chi_3(\pi)$.

\subsection{Finding $[\rho_0]$ in the deformation variety}
The  specific representation $\rho_0$ we consider is defined by $\rho_0(a)=S$ and $\rho_0(b)=T$
where $S$ and $T$ are the order three elements in SL(3,$\C$) given by \eqref{eq:defST} in Section \ref{ss:upper-bound}.
We have seen in Proposition \ref{pr:unique-deco} that $\rho_0$ admits a unique decoration. In Table \ref{flags}, 
we provide the flags associated by this decoration to the six vertices of the octahedron described in Section \ref{ss:octa}.
Note that $\rho_0$ maps every stabilizer of a vertex of the octahedron to a cyclic group. The flag associated
to this vertex is in fact invariant under the image by $\rho_0$ of the stabilizer.

%In order to compute the local dimension of the character variety, we identify 
%a decoration of $\rho_0$ as a point in the previous defined deformation variety.
%We shall decorate the representation $\rho_0$, i.e. 
%choose  for each cusp in the universal covering of the Whitehead Link Complement
% a flag invariant by the stabilizer in $\pi$ of the cusp. Equivalently,
% we choose for each vertex of the octahedron a flag invariant by the stabilizer of this vertex.
%For each vertex of the octahedron, we describe in Table \ref{flags} an element in its stabilizer in $\pi$ and a 
%flag invariant by its image under $\rho_0$. 
%Note that as $\rho_0$ maps every stabilizer of a vertex to a cyclic group, the flag is in fact invariant under $\rho_0$ of 
%the whole stabilizer.
\medskip

\begin{table}[ht]
\begin{center}
\begin{tabular}{c|c|c}
Vertex & Generator of its stabilizer in the image of $\rho_0$ & Invariant flag \\
&&\\
\hline
&&\\
$\infty$ &  $ST^{-1}$ & $F_{\infty}  : \left[ \begin{smallmatrix} 1 \\ 0\\0\end{smallmatrix} \right], [0,0,1]$ \\
&&\\
\hline
&&\\
$0$ &  $ST^{-1}S$ & $F_0 : \left[ \begin{smallmatrix} 1 \\ -\frac{3\sqrt{3} + i\sqrt{5}}{4}\\-1\end{smallmatrix} \right], [1, \frac{3\sqrt{3} - i\sqrt{5}}{4},-1]$\\
&&\\
\hline
&&\\
$i$ &  $S^{-1}T^{-1}$ & $F_i : \left[ \begin{smallmatrix} 1 \\ -\frac{\sqrt{3} + i\sqrt{5}}{4}\\\frac{-1+i\sqrt{15}}{4}\end{smallmatrix} \right], [\frac{1+i\sqrt{15}}{4}, \frac{\sqrt{3} - i\sqrt{5}}{4},-1]$\\
&&\\
\hline
&&\\
$-1+i$ &  $T^{-1}ST^{-1}$ & $F_{-1+i} : \left[ \begin{smallmatrix} 1 \\ \frac{3\sqrt{3} - i\sqrt{5}}{4}\\-1\end{smallmatrix} \right], [1, -\frac{3\sqrt{3} + i\sqrt{5}}{4},-1]$\\
&&\\
\hline
&&\\
$-i$ &  $T^{-1}S^{-1}$ & $F_{-i} : \left[ \begin{smallmatrix} 1 \\ \frac{\sqrt{3} - i\sqrt{5}}{4}\\-\frac{1+i\sqrt{15}}{4}\end{smallmatrix} \right],[\frac{1-i\sqrt{15}}{4},-\frac{\sqrt{3} + i\sqrt{5}}{4},-1]$\\
&&\\
\hline
&&\\
$\frac{-1+i}{2}$ &  $TS$ & $F_{\frac{-1+i}{2}} : \left[ \begin{smallmatrix} 0 \\ 0\\1\end{smallmatrix} \right], [1,0,0]$ \\
&&\\
\end{tabular}
\end{center}
\caption{The unique decoration of $\rho_0$}\label{flags}
\end{table}
{As a result, each tetrahedron $(\Delta_\nu)_{1\leq \nu\leq 4}$ is decorated by four flags. For instance, the tetrahedron
$\Delta_0=(i,0,\frac{-1+i}{2},\infty)$ is decorated by $(F_{i},F_0,F_{\frac{-1+i}{2}},F_\infty)$ and the other three tetrahedra 
are decorated in a similar way (the tetrahedra are listed in Section \ref{ss:triangulation}). It is a simple calculation to compute 
the cross-ratios associated to these flags as explained in Section \ref{ss:defor}.
Table \ref{coord} displays, for each tetrahedron, the values of coordinates $z_{12}$, $z_{21}$, $z_{34}$, $z_{43}$. 
The values of the other coordinates can be deduced from them using the internal relations \eqref{IR}.}\\
% T_0 est le tetraèdre 1 dans la feuille de calcul Sage
% 
\begin{table}[ht]
\begin{center}
\begin{tabular}{c|c|c|c|c}
Tetrahedron & $z_{12}$ & $z_{21}$ & $z_{34}$ & $z_{43}$\\
&&&&\\
\hline
&&&&\\
$\Delta_0$ & $\frac{7+i\sqrt{15}}{4}$ & $\frac{-1-i\sqrt{15}}{8}$ &
                 $\frac{7-i\sqrt{15}}{4}$ & $\frac{-1+i\sqrt{15}}{8}$ \\
&&&&\\
\hline
&&&&\\
$\Delta_1$ & $\frac{-1 + i\sqrt{15}}{8}$ & $\frac{7 - i\sqrt{15}}{4}$ &
                 $\frac{-1 - i\sqrt{15}}{8}$ & $\frac{7+i\sqrt{15}}{4}$  \\
&&&&\\                 
                 
\hline
&&&&\\
$\Delta_2$ & $\frac{7 - i\sqrt{15}}{4}$ & $\frac{-1 + i\sqrt{15}}{8}$ &
                 $\frac{-1-i\sqrt{15}}{8}$ & $\frac{7+i\sqrt{15}}{4}$ \\
&&&&\\                 
                 
\hline
&&&&\\
$\Delta_3$ & $\frac{-1 - i\sqrt{15}}{8}$ & $\frac{7+i\sqrt{15}}{4}$ &
                 $\frac{-1 + i\sqrt{15}}{8}$ & $\frac{7-i\sqrt{15}}{4}$ \\
\end{tabular}
\end{center}
\caption{Coordinates for $[\rho_0]$ in the deformation variety}\label{coord}
\end{table}
{\begin{rem}
Note the high degree of symmetry of the considered decorated representation: the tetrahedra are
all the same up to the action of $\SL(3,\C)$ and a renumbering of the vertices.
\end{rem}}
\subsection{Computation of the Zariski tangent space: proof of Proposition \ref{pr:upper-bound}.}
The deformation variety is the (algebraic) set of all tuples of 48 complex numbers satisfying both the internal relations and 
the gluing equations (compare to \cite[Section 3]{BFGKR}). In other words it is the intersection of the inverse images

$$\mathrm{IR}^{-1}(1,\ldots,1)\cap \mathrm{GE}^{-1}(1,\ldots,1),$$
where the two maps $\mathrm{IR}$ and $\mathrm{GE}$ are defined by
\goodbreak
\begin{itemize}
\item $\mathrm{IR} \: : \: (\C\setminus\{0,1\})^{48} \to (\C^*)^{32}$ is the map representing the internal relations \eqref{IR}: it sends a collection 
$(z_{ij}(\Delta_{\nu}))$ (for every half-edge $ij$ and $1\leq \nu \leq 4$) to the collection of $32$ complex numbers given by:
$$(-z_{ij}(\Delta_\nu)z_{ik}(\Delta_\nu)z_{il}(\Delta_\nu), \quad z_{ik}(\Delta_\nu)(1-z_{ij}(\Delta_\nu))).$$
\item $\mathrm{GE} \: : \: (\C\setminus\{0,1\})^{48} \to (\C^*)^{16}$ is the collection of left-hand sides in the gluing equations of table \ref{GE}.
\end{itemize}
We denote by $(\mathrm{IR},\mathrm{GE})$ the map $(\C\setminus\{0,1\})^{48} \to (\C^*)^{48}$ given by 
the previous two maps.  The Zariski tangent space to the deformation variety is the kernel of the 
tangent map to $(\mathrm{IR},\mathrm{GE})$. As those maps are mostly monomial, we choose to
write their tangent maps in the following basis of tangent spaces at the source and target. We take, for each 
coordinate $z$, the vector field $z\frac{\partial}{\partial z}$.
In these basis, the tangent map to a function 
$\phi : (\C^*)^{48}\to \C^*$
has entries of the form $$\frac{z}{\phi} \frac{\partial \phi}{\partial z}.$$ 
This follows from the following elementary lemma:
\begin{lem}
Let $f : \C^* \to \C^*$ be a differentiable function. In the basis 
$z\frac{\partial}{\partial z}$, 
the tangent map $Tf$ at $a\in \C^*$ has coordinate $\frac{a}{f(a)}\frac{\partial f}{\partial z}$. 
\end{lem}
\begin{proof}
In usual coordinates, by definition of partial derivative, the tangent map $T_af$ at a point $a$ maps 
$$ v\in T_a\C^* \quad \textrm{ to } \quad w = \frac{\partial f}{\partial z} v \in T_{f(a)}\C^*.$$
The basis change from $\frac{\partial}{\partial z}$ to $z\frac{\partial}{\partial z}$ in both tangent spaces transforms $v$ into $u = \frac{v}{a}$ and $w$ into $\frac{w}{f(a)}$. Hence, in new basis,
$u$ is sent to $$\frac{w}{f(a)}=\frac{1}{f(a)}\frac{\partial f}{\partial z} v = \frac{a}{f(a)}\frac{\partial f}{\partial z} u.$$
\end{proof}
 We will apply this lemma to each coordinate of the map $(\mathrm{IR},\mathrm{GE})$ to obtain a matrix for the tangent map to 
$(\mathrm{IR},\mathrm{GE})$ at the point $[\rho_0]$. This matrix, denoted $J$, has size $48\times48$ and is 
depicted in Table \ref{table:J} (see Remark \ref{rem:order} below). To construct $J$, we have to deal with two kinds of  functions, depending on the equations 
that form the the maps IR and GE : monomial maps or maps of the form $z_{ik}(\Delta_\nu)(1-z_{ij}(\Delta_\nu))$.
\begin{itemize}
\item If $f$ is a monomial map, its tangent map has integer entries equal to the 
exponent of the relative variable. As an example, this formula applied to the map 
$-z_{12}(\Delta_0)z_{13}(\Delta_0)z_{14}(\Delta_0)$ gives all entries equal to $0$ 
except for those associated to the variables  $z_{12}(\Delta_0)$, $z_{13}(\Delta_0)$ and 
 $z_{14}(\Delta_0)$ which give three entries equal to $1$. The same phenomenon appears for each of the first
sixteen rows of the matrix $J$. The gluing equations are also monomials (see Table \ref{GE}), but involve 
more variables. These correspond to lines 33 to 48 of the matrix $J$, that have all their coefficients equal to $0$ except for 4, 6 or 8 of them that are equal to $1$.
\item if $f$ has the form\footnote{We drop here the 
indication of the tetrahedron $\Delta_\nu$ for 
$0\leq \nu \leq 3$ in order to simplify the notations.}  $z_{ik}(1-z_{ij})$, its tangent map
 has every entry equal to $0$ except the ones corresponding to $z_{ij}$ and $z_{ik}$. 
Those two are respectively $-\frac{z_{ij}}{1-z_{ij}}$ and $1$. 
Note that, at a point satisfying the internal 
relations, we have the additional relation $ -\frac{z_{ij}}{1-z_{ij}}=z_{il}$ 
(see also the computation in 
\cite[Section 5]{BFGKR}, especially Lemma 5.3). Hence the entries for such a map are $0$, $1$ or $z_{il}$. Those appear in rows 
17 to 32 of the matrix $J$ displayed in Table \ref{table:J}. 
\end{itemize}
%Let $J$ be the matrix, in these basis, of the tangent map to
%$(\mathrm{IR},\mathrm{GE})$ at $[\rho_0]$.
We see thus that $J$ has entries either integer or of the form $z_{il}(\Delta_\nu)$.
Note moreover that the last 16 rows, corresponding to the gluing equations, can be accessed directly 
by SnapPy \cite{SnapPy} under SageMath \cite{sage}: it is the Neumann-Zagier datum. This part of $J$ is 
directly given by the commands: 
\begin{verbatim}
import snappy; 
Triangulation("5^2_1").gluing_equations_pgl(3,equation_type='non_peripheral').matrix  
\end{verbatim}
%It compose the last 16 lines in the matrix $J$ displayed in Table \ref{table:J}.

The next step is to compute the kernel of $J$. As all entries are in the 
number field $\mathbb Q[i,\sqrt{3},\sqrt{5}]$, a computer algebra system  such as Sage computes it exactly. 
As a result, the dimension of this kernel is $4$ (see the Sage notebook \cite{Notebook}). We deduce that the 
dimension of the Zariski tangent space at the decoration of $[\rho_0]$ to the deformation variety is $4$.

\begin{rem}\label{rem:order}
To write the matrix $J$, we choose the same order on the variables $z_{ij}(\Delta_\nu)$ as SnapPy does.  
As the precise order it is not very enlightening, we omit this discussion here. A change of order on the 
variables amounts to a permutation of the columns of $J$, which does not affect the dimension of its kernel.
\end{rem}

Note that at $[\rho_0]$, the two subgroups generated by the pairs $(\rho_0(l_i),\rho_0(m_i))$ for $i=1,\, 2$ are 
regular unipotent: there is only one invariant flag for each one. As noted before, it implies that, locally the 
holonomy map between the deformation variety and the actual character variety is a finite 
ramified covering. This concludes the proof of Proposition \ref{pr:upper-bound}: 
the Zariski tangent space to the character variety at $[\rho_0]$ also has dimension $4$.

\begin{table}[ht]
\begin{center}
\scalebox{.34}{%\rotatebox{-90}
{
$
\mathbf{\begin{array}{lcccccccccccccccccccccccccccccccccccccccccccccccc}
\scalebox{2}{\textrm{\black Row 1:}}&\scalebox{2}{$\mathbf{1}$} & \scalebox{2}{$\mathbf{1}$} & \scalebox{2}{$\mathbf{1}$} & {\black 0} & {\black 0} & {\black 0} & {\black 0} & {\black 0} & {\black 0} & {\black 0} & {\black 0} & {\black 0} & {\black 0} & {\black 0} & {\black 0} & {\black 0} & {\black 0} & {\black 0} & {\black 0} & {\black 0} & {\black 0} & {\black 0} & {\black 0} & {\black 0} & {\black 0} & {\black 0} & {\black 0} & {\black 0} & {\black 0} & {\black 0} & {\black 0} & {\black 0} & {\black 0} & {\black 0} & {\black 0} & {\black 0} & {\black 0} & {\black 0} & {\black 0} & {\black 0} & {\black 0} & {\black 0} & {\black 0} & {\black 0} & {\black 0} & {\black 0} & {\black 0} & {\black 0} \\
&{\black 0} & {\black 0} & {\black 0} & \scalebox{2}{$\mathbf{1}$} & \scalebox{2}{$\mathbf{1}$} & \scalebox{2}{$\mathbf{1}$} & {\black 0} & {\black 0} & {\black 0} & {\black 0} & {\black 0} & {\black 0} & {\black 0} & {\black 0} & {\black 0} & {\black 0} & {\black 0} & {\black 0} & {\black 0} & {\black 0} & {\black 0} & {\black 0} & {\black 0} & {\black 0} & {\black 0} & {\black 0} & {\black 0} & {\black 0} & {\black 0} & {\black 0} & {\black 0} & {\black 0} & {\black 0} & {\black 0} & {\black 0} & {\black 0} & {\black 0} & {\black 0} & {\black 0} & {\black 0} & {\black 0} & {\black 0} & {\black 0} & {\black 0} & {\black 0} & {\black 0} & {\black 0} & {\black 0} \\
&{\black 0} & {\black 0} & {\black 0} & {\black 0} & {\black 0} & {\black 0} & \scalebox{2}{$\mathbf{1}$} & \scalebox{2}{$\mathbf{1}$} & \scalebox{2}{$\mathbf{1}$} & {\black 0} & {\black 0} & {\black 0} & {\black 0} & {\black 0} & {\black 0} & {\black 0} & {\black 0} & {\black 0} & {\black 0} & {\black 0} & {\black 0} & {\black 0} & {\black 0} & {\black 0} & {\black 0} & {\black 0} & {\black 0} & {\black 0} & {\black 0} & {\black 0} & {\black 0} & {\black 0} & {\black 0} & {\black 0} & {\black 0} & {\black 0} & {\black 0} & {\black 0} & {\black 0} & {\black 0} & {\black 0} & {\black 0} & {\black 0} & {\black 0} & {\black 0} & {\black 0} & {\black 0} & {\black 0} \\
&{\black 0} & {\black 0} & {\black 0} & {\black 0} & {\black 0} & {\black 0} & {\black 0} & {\black 0} & {\black 0} & \scalebox{2}{$\mathbf{1}$} & \scalebox{2}{$\mathbf{1}$} & \scalebox{2}{$\mathbf{1}$} & {\black 0} & {\black 0} & {\black 0} & {\black 0} & {\black 0} & {\black 0} & {\black 0} & {\black 0} & {\black 0} & {\black 0} & {\black 0} & {\black 0} & {\black 0} & {\black 0} & {\black 0} & {\black 0} & {\black 0} & {\black 0} & {\black 0} & {\black 0} & {\black 0} & {\black 0} & {\black 0} & {\black 0} & {\black 0} & {\black 0} & {\black 0} & {\black 0} & {\black 0} & {\black 0} & {\black 0} & {\black 0} & {\black 0} & {\black 0} & {\black 0} & {\black 0} \\
&{\black 0} & {\black 0} & {\black 0} & {\black 0} & {\black 0} & {\black 0} & {\black 0} & {\black 0} & {\black 0} & {\black 0} & {\black 0} & {\black 0} & \scalebox{2}{$\mathbf{1}$} & \scalebox{2}{$\mathbf{1}$} & \scalebox{2}{$\mathbf{1}$} & {\black 0} & {\black 0} & {\black 0} & {\black 0} & {\black 0} & {\black 0} & {\black 0} & {\black 0} & {\black 0} & {\black 0} & {\black 0} & {\black 0} & {\black 0} & {\black 0} & {\black 0} & {\black 0} & {\black 0} & {\black 0} & {\black 0} & {\black 0} & {\black 0} & {\black 0} & {\black 0} & {\black 0} & {\black 0} & {\black 0} & {\black 0} & {\black 0} & {\black 0} & {\black 0} & {\black 0} & {\black 0} & {\black 0} \\
&{\black 0} & {\black 0} & {\black 0} & {\black 0} & {\black 0} & {\black 0} & {\black 0} & {\black 0} & {\black 0} & {\black 0} & {\black 0} & {\black 0} & {\black 0} & {\black 0} & {\black 0} & \scalebox{2}{$\mathbf{1}$} & \scalebox{2}{$\mathbf{1}$} & \scalebox{2}{$\mathbf{1}$} & {\black 0} & {\black 0} & {\black 0} & {\black 0} & {\black 0} & {\black 0} & {\black 0} & {\black 0} & {\black 0} & {\black 0} & {\black 0} & {\black 0} & {\black 0} & {\black 0} & {\black 0} & {\black 0} & {\black 0} & {\black 0} & {\black 0} & {\black 0} & {\black 0} & {\black 0} & {\black 0} & {\black 0} & {\black 0} & {\black 0} & {\black 0} & {\black 0} & {\black 0} & {\black 0} \\
&{\black 0} & {\black 0} & {\black 0} & {\black 0} & {\black 0} & {\black 0} & {\black 0} & {\black 0} & {\black 0} & {\black 0} & {\black 0} & {\black 0} & {\black 0} & {\black 0} & {\black 0} & {\black 0} & {\black 0} & {\black 0} & \scalebox{2}{$\mathbf{1}$} & \scalebox{2}{$\mathbf{1}$} & \scalebox{2}{$\mathbf{1}$} & {\black 0} & {\black 0} & {\black 0} & {\black 0} & {\black 0} & {\black 0} & {\black 0} & {\black 0} & {\black 0} & {\black 0} & {\black 0} & {\black 0} & {\black 0} & {\black 0} & {\black 0} & {\black 0} & {\black 0} & {\black 0} & {\black 0} & {\black 0} & {\black 0} & {\black 0} & {\black 0} & {\black 0} & {\black 0} & {\black 0} & {\black 0} \\
&{\black 0} & {\black 0} & {\black 0} & {\black 0} & {\black 0} & {\black 0} & {\black 0} & {\black 0} & {\black 0} & {\black 0} & {\black 0} & {\black 0} & {\black 0} & {\black 0} & {\black 0} & {\black 0} & {\black 0} & {\black 0} & {\black 0} & {\black 0} & {\black 0} & \scalebox{2}{$\mathbf{1}$} & \scalebox{2}{$\mathbf{1}$} & \scalebox{2}{$\mathbf{1}$} & {\black 0} & {\black 0} & {\black 0} & {\black 0} & {\black 0} & {\black 0} & {\black 0} & {\black 0} & {\black 0} & {\black 0} & {\black 0} & {\black 0} & {\black 0} & {\black 0} & {\black 0} & {\black 0} & {\black 0} & {\black 0} & {\black 0} & {\black 0} & {\black 0} & {\black 0} & {\black 0} & {\black 0} \\
&{\black 0} & {\black 0} & {\black 0} & {\black 0} & {\black 0} & {\black 0} & {\black 0} & {\black 0} & {\black 0} & {\black 0} & {\black 0} & {\black 0} & {\black 0} & {\black 0} & {\black 0} & {\black 0} & {\black 0} & {\black 0} & {\black 0} & {\black 0} & {\black 0} & {\black 0} & {\black 0} & {\black 0} & \scalebox{2}{$\mathbf{1}$} & \scalebox{2}{$\mathbf{1}$} & \scalebox{2}{$\mathbf{1}$} & {\black 0} & {\black 0} & {\black 0} & {\black 0} & {\black 0} & {\black 0} & {\black 0} & {\black 0} & {\black 0} & {\black 0} & {\black 0} & {\black 0} & {\black 0} & {\black 0} & {\black 0} & {\black 0} & {\black 0} & {\black 0} & {\black 0} & {\black 0} & {\black 0} \\
&{\black 0} & {\black 0} & {\black 0} & {\black 0} & {\black 0} & {\black 0} & {\black 0} & {\black 0} & {\black 0} & {\black 0} & {\black 0} & {\black 0} & {\black 0} & {\black 0} & {\black 0} & {\black 0} & {\black 0} & {\black 0} & {\black 0} & {\black 0} & {\black 0} & {\black 0} & {\black 0} & {\black 0} & {\black 0} & {\black 0} & {\black 0} & \scalebox{2}{$\mathbf{1}$} & \scalebox{2}{$\mathbf{1}$} & \scalebox{2}{$\mathbf{1}$} & {\black 0} & {\black 0} & {\black 0} & {\black 0} & {\black 0} & {\black 0} & {\black 0} & {\black 0} & {\black 0} & {\black 0} & {\black 0} & {\black 0} & {\black 0} & {\black 0} & {\black 0} & {\black 0} & {\black 0} & {\black 0} \\
&{\black 0} & {\black 0} & {\black 0} & {\black 0} & {\black 0} & {\black 0} & {\black 0} & {\black 0} & {\black 0} & {\black 0} & {\black 0} & {\black 0} & {\black 0} & {\black 0} & {\black 0} & {\black 0} & {\black 0} & {\black 0} & {\black 0} & {\black 0} & {\black 0} & {\black 0} & {\black 0} & {\black 0} & {\black 0} & {\black 0} & {\black 0} & {\black 0} & {\black 0} & {\black 0} & \scalebox{2}{$\mathbf{1}$} & \scalebox{2}{$\mathbf{1}$} & \scalebox{2}{$\mathbf{1}$} & {\black 0} & {\black 0} & {\black 0} & {\black 0} & {\black 0} & {\black 0} & {\black 0} & {\black 0} & {\black 0} & {\black 0} & {\black 0} & {\black 0} & {\black 0} & {\black 0} & {\black 0} \\
&{\black 0} & {\black 0} & {\black 0} & {\black 0} & {\black 0} & {\black 0} & {\black 0} & {\black 0} & {\black 0} & {\black 0} & {\black 0} & {\black 0} & {\black 0} & {\black 0} & {\black 0} & {\black 0} & {\black 0} & {\black 0} & {\black 0} & {\black 0} & {\black 0} & {\black 0} & {\black 0} & {\black 0} & {\black 0} & {\black 0} & {\black 0} & {\black 0} & {\black 0} & {\black 0} & {\black 0} & {\black 0} & {\black 0} & \scalebox{2}{$\mathbf{1}$} & \scalebox{2}{$\mathbf{1}$} & \scalebox{2}{$\mathbf{1}$} & {\black 0} & {\black 0} & {\black 0} & {\black 0} & {\black 0} & {\black 0} & {\black 0} & {\black 0} & {\black 0} & {\black 0} & {\black 0} & {\black 0} \\
&{\black 0} & {\black 0} & {\black 0} & {\black 0} & {\black 0} & {\black 0} & {\black 0} & {\black 0} & {\black 0} & {\black 0} & {\black 0} & {\black 0} & {\black 0} & {\black 0} & {\black 0} & {\black 0} & {\black 0} & {\black 0} & {\black 0} & {\black 0} & {\black 0} & {\black 0} & {\black 0} & {\black 0} & {\black 0} & {\black 0} & {\black 0} & {\black 0} & {\black 0} & {\black 0} & {\black 0} & {\black 0} & {\black 0} & {\black 0} & {\black 0} & {\black 0} & \scalebox{2}{$\mathbf{1}$} & \scalebox{2}{$\mathbf{1}$} & \scalebox{2}{$\mathbf{1}$} & {\black 0} & {\black 0} & {\black 0} & {\black 0} & {\black 0} & {\black 0} & {\black 0} & {\black 0} & {\black 0} \\
&{\black 0} & {\black 0} & {\black 0} & {\black 0} & {\black 0} & {\black 0} & {\black 0} & {\black 0} & {\black 0} & {\black 0} & {\black 0} & {\black 0} & {\black 0} & {\black 0} & {\black 0} & {\black 0} & {\black 0} & {\black 0} & {\black 0} & {\black 0} & {\black 0} & {\black 0} & {\black 0} & {\black 0} & {\black 0} & {\black 0} & {\black 0} & {\black 0} & {\black 0} & {\black 0} & {\black 0} & {\black 0} & {\black 0} & {\black 0} & {\black 0} & {\black 0} & {\black 0} & {\black 0} & {\black 0} & \scalebox{2}{$\mathbf{1}$} & \scalebox{2}{$\mathbf{1}$} & \scalebox{2}{$\mathbf{1}$} & {\black 0} & {\black 0} & {\black 0} & {\black 0} & {\black 0} & {\black 0} \\
&{\black 0} & {\black 0} & {\black 0} & {\black 0} & {\black 0} & {\black 0} & {\black 0} & {\black 0} & {\black 0} & {\black 0} & {\black 0} & {\black 0} & {\black 0} & {\black 0} & {\black 0} & {\black 0} & {\black 0} & {\black 0} & {\black 0} & {\black 0} & {\black 0} & {\black 0} & {\black 0} & {\black 0} & {\black 0} & {\black 0} & {\black 0} & {\black 0} & {\black 0} & {\black 0} & {\black 0} & {\black 0} & {\black 0} & {\black 0} & {\black 0} & {\black 0} & {\black 0} & {\black 0} & {\black 0} & {\black 0} & {\black 0} & {\black 0} & \scalebox{2}{$\mathbf{1}$} & \scalebox{2}{$\mathbf{1}$} & \scalebox{2}{$\mathbf{1}$} & {\black 0} & {\black 0} & {\black 0} \\
\scalebox{2}{\textrm{\black Row 16:}} &{\black 0} & {\black 0} & {\black 0} & {\black 0} & {\black 0} & {\black 0} & {\black 0} & {\black 0} & {\black 0} & {\black 0} & {\black 0} & {\black 0} & {\black 0} & {\black 0} & {\black 0} & {\black 0} & {\black 0} & {\black 0} & {\black 0} & {\black 0} & {\black 0} & {\black 0} & {\black 0} & {\black 0} & {\black 0} & {\black 0} & {\black 0} & {\black 0} & {\black 0} & {\black 0} & {\black 0} & {\black 0} & {\black 0} & {\black 0} & {\black 0} & {\black 0} & {\black 0} & {\black 0} & {\black 0} & {\black 0} & {\black 0} & {\black 0} & {\black 0} & {\black 0} & {\black 0} & \scalebox{2}{$\mathbf{1}$} & \scalebox{2}{$\mathbf{1}$} & \scalebox{2}{$\mathbf{1}$} \\
\scalebox{2}{\textrm{\black Row 17:}} &\scalebox{2}{$\mathbf{1}$} & {\black 0} &  \scalebox{2}{$\mathbf{\bar x}$}   & {\black 0} & {\black 0} & {\black 0} & {\black 0} & {\black 0} & {\black 0} & {\black 0} & {\black 0} & {\black 0} & {\black 0} & {\black 0} & {\black 0} & {\black 0} & {\black 0} & {\black 0} & {\black 0} & {\black 0} & {\black 0} & {\black 0} & {\black 0} & {\black 0} & {\black 0} & {\black 0} & {\black 0} & {\black 0} & {\black 0} & {\black 0} & {\black 0} & {\black 0} & {\black 0} & {\black 0} & {\black 0} & {\black 0} & {\black 0} & {\black 0} & {\black 0} & {\black 0} & {\black 0} & {\black 0} & {\black 0} & {\black 0} & {\black 0} & {\black 0} & {\black 0} & {\black 0} \\
&{\black 0} & {\black 0} & {\black 0} & \scalebox{2}{$\mathbf{1}$} & {\black 0} &  \scalebox{2}{$\mathbf{\bar y}$}   & {\black 0} & {\black 0} & {\black 0} & {\black 0} & {\black 0} & {\black 0} & {\black 0} & {\black 0} & {\black 0} & {\black 0} & {\black 0} & {\black 0} & {\black 0} & {\black 0} & {\black 0} & {\black 0} & {\black 0} & {\black 0} & {\black 0} & {\black 0} & {\black 0} & {\black 0} & {\black 0} & {\black 0} & {\black 0} & {\black 0} & {\black 0} & {\black 0} & {\black 0} & {\black 0} & {\black 0} & {\black 0} & {\black 0} & {\black 0} & {\black 0} & {\black 0} & {\black 0} & {\black 0} & {\black 0} & {\black 0} & {\black 0} & {\black 0} \\
&{\black 0} & {\black 0} & {\black 0} & {\black 0} & {\black 0} & {\black 0} & \scalebox{2}{$\mathbf{1}$} & {\black 0} & \ \scalebox{2}{$\mathbf{x}$}   & {\black 0} & {\black 0} & {\black 0} & {\black 0} & {\black 0} & {\black 0} & {\black 0} & {\black 0} & {\black 0} & {\black 0} & {\black 0} & {\black 0} & {\black 0} & {\black 0} & {\black 0} & {\black 0} & {\black 0} & {\black 0} & {\black 0} & {\black 0} & {\black 0} & {\black 0} & {\black 0} & {\black 0} & {\black 0} & {\black 0} & {\black 0} & {\black 0} & {\black 0} & {\black 0} & {\black 0} & {\black 0} & {\black 0} & {\black 0} & {\black 0} & {\black 0} & {\black 0} & {\black 0} & {\black 0} \\
&{\black 0} & {\black 0} & {\black 0} & {\black 0} & {\black 0} & {\black 0} & {\black 0} & {\black 0} & {\black 0} & \scalebox{2}{$\mathbf{1}$} & {\black 0} & \ \scalebox{2}{$\mathbf{y}$}   & {\black 0} & {\black 0} & {\black 0} & {\black 0} & {\black 0} & {\black 0} & {\black 0} & {\black 0} & {\black 0} & {\black 0} & {\black 0} & {\black 0} & {\black 0} & {\black 0} & {\black 0} & {\black 0} & {\black 0} & {\black 0} & {\black 0} & {\black 0} & {\black 0} & {\black 0} & {\black 0} & {\black 0} & {\black 0} & {\black 0} & {\black 0} & {\black 0} & {\black 0} & {\black 0} & {\black 0} & {\black 0} & {\black 0} & {\black 0} & {\black 0} & {\black 0} \\
&{\black 0} & {\black 0} & {\black 0} & {\black 0} & {\black 0} & {\black 0} & {\black 0} & {\black 0} & {\black 0} & {\black 0} & {\black 0} & {\black 0} & \scalebox{2}{$\mathbf{1}$} & {\black 0} &  \scalebox{2}{$\mathbf{\bar x}$}   & {\black 0} & {\black 0} & {\black 0} & {\black 0} & {\black 0} & {\black 0} & {\black 0} & {\black 0} & {\black 0} & {\black 0} & {\black 0} & {\black 0} & {\black 0} & {\black 0} & {\black 0} & {\black 0} & {\black 0} & {\black 0} & {\black 0} & {\black 0} & {\black 0} & {\black 0} & {\black 0} & {\black 0} & {\black 0} & {\black 0} & {\black 0} & {\black 0} & {\black 0} & {\black 0} & {\black 0} & {\black 0} & {\black 0} \\
&{\black 0} & {\black 0} & {\black 0} & {\black 0} & {\black 0} & {\black 0} & {\black 0} & {\black 0} & {\black 0} & {\black 0} & {\black 0} & {\black 0} & {\black 0} & {\black 0} & {\black 0} & \scalebox{2}{$\mathbf{1}$} & {\black 0} &  \scalebox{2}{$\mathbf{\bar y}$}   & {\black 0} & {\black 0} & {\black 0} & {\black 0} & {\black 0} & {\black 0} & {\black 0} & {\black 0} & {\black 0} & {\black 0} & {\black 0} & {\black 0} & {\black 0} & {\black 0} & {\black 0} & {\black 0} & {\black 0} & {\black 0} & {\black 0} & {\black 0} & {\black 0} & {\black 0} & {\black 0} & {\black 0} & {\black 0} & {\black 0} & {\black 0} & {\black 0} & {\black 0} & {\black 0} \\
&{\black 0} & {\black 0} & {\black 0} & {\black 0} & {\black 0} & {\black 0} & {\black 0} & {\black 0} & {\black 0} & {\black 0} & {\black 0} & {\black 0} & {\black 0} & {\black 0} & {\black 0} & {\black 0} & {\black 0} & {\black 0} & \scalebox{2}{$\mathbf{1}$} & {\black 0} & \ \scalebox{2}{$\mathbf{x}$}   & {\black 0} & {\black 0} & {\black 0} & {\black 0} & {\black 0} & {\black 0} & {\black 0} & {\black 0} & {\black 0} & {\black 0} & {\black 0} & {\black 0} & {\black 0} & {\black 0} & {\black 0} & {\black 0} & {\black 0} & {\black 0} & {\black 0} & {\black 0} & {\black 0} & {\black 0} & {\black 0} & {\black 0} & {\black 0} & {\black 0} & {\black 0} \\
&{\black 0} & {\black 0} & {\black 0} & {\black 0} & {\black 0} & {\black 0} & {\black 0} & {\black 0} & {\black 0} & {\black 0} & {\black 0} & {\black 0} & {\black 0} & {\black 0} & {\black 0} & {\black 0} & {\black 0} & {\black 0} & {\black 0} & {\black 0} & {\black 0} & \scalebox{2}{$\mathbf{1}$} & {\black 0} & \ \scalebox{2}{$\mathbf{y}$}   & {\black 0} & {\black 0} & {\black 0} & {\black 0} & {\black 0} & {\black 0} & {\black 0} & {\black 0} & {\black 0} & {\black 0} & {\black 0} & {\black 0} & {\black 0} & {\black 0} & {\black 0} & {\black 0} & {\black 0} & {\black 0} & {\black 0} & {\black 0} & {\black 0} & {\black 0} & {\black 0} & {\black 0} \\
&{\black 0} & {\black 0} & {\black 0} & {\black 0} & {\black 0} & {\black 0} & {\black 0} & {\black 0} & {\black 0} & {\black 0} & {\black 0} & {\black 0} & {\black 0} & {\black 0} & {\black 0} & {\black 0} & {\black 0} & {\black 0} & {\black 0} & {\black 0} & {\black 0} & {\black 0} & {\black 0} & {\black 0} & \scalebox{2}{$\mathbf{1}$} & {\black 0} & \ \scalebox{2}{$\mathbf{x}$}   & {\black 0} & {\black 0} & {\black 0} & {\black 0} & {\black 0} & {\black 0} & {\black 0} & {\black 0} & {\black 0} & {\black 0} & {\black 0} & {\black 0} & {\black 0} & {\black 0} & {\black 0} & {\black 0} & {\black 0} & {\black 0} & {\black 0} & {\black 0} & {\black 0} \\
&{\black 0} & {\black 0} & {\black 0} & {\black 0} & {\black 0} & {\black 0} & {\black 0} & {\black 0} & {\black 0} & {\black 0} & {\black 0} & {\black 0} & {\black 0} & {\black 0} & {\black 0} & {\black 0} & {\black 0} & {\black 0} & {\black 0} & {\black 0} & {\black 0} & {\black 0} & {\black 0} & {\black 0} & {\black 0} & {\black 0} & {\black 0} & \scalebox{2}{$\mathbf{1}$} & {\black 0} & \ \scalebox{2}{$\mathbf{y}$}   & {\black 0} & {\black 0} & {\black 0} & {\black 0} & {\black 0} & {\black 0} & {\black 0} & {\black 0} & {\black 0} & {\black 0} & {\black 0} & {\black 0} & {\black 0} & {\black 0} & {\black 0} & {\black 0} & {\black 0} & {\black 0} \\
&{\black 0} & {\black 0} & {\black 0} & {\black 0} & {\black 0} & {\black 0} & {\black 0} & {\black 0} & {\black 0} & {\black 0} & {\black 0} & {\black 0} & {\black 0} & {\black 0} & {\black 0} & {\black 0} & {\black 0} & {\black 0} & {\black 0} & {\black 0} & {\black 0} & {\black 0} & {\black 0} & {\black 0} & {\black 0} & {\black 0} & {\black 0} & {\black 0} & {\black 0} & {\black 0} & \scalebox{2}{$\mathbf{1}$} & {\black 0} &  \scalebox{2}{$\mathbf{\bar x}$}   & {\black 0} & {\black 0} & {\black 0} & {\black 0} & {\black 0} & {\black 0} & {\black 0} & {\black 0} & {\black 0} & {\black 0} & {\black 0} & {\black 0} & {\black 0} & {\black 0} & {\black 0} \\
&{\black 0} & {\black 0} & {\black 0} & {\black 0} & {\black 0} & {\black 0} & {\black 0} & {\black 0} & {\black 0} & {\black 0} & {\black 0} & {\black 0} & {\black 0} & {\black 0} & {\black 0} & {\black 0} & {\black 0} & {\black 0} & {\black 0} & {\black 0} & {\black 0} & {\black 0} & {\black 0} & {\black 0} & {\black 0} & {\black 0} & {\black 0} & {\black 0} & {\black 0} & {\black 0} & {\black 0} & {\black 0} & {\black 0} & \scalebox{2}{$\mathbf{1}$} & {\black 0} &  \scalebox{2}{$\mathbf{\bar y}$}   & {\black 0} & {\black 0} & {\black 0} & {\black 0} & {\black 0} & {\black 0} & {\black 0} & {\black 0} & {\black 0} & {\black 0} & {\black 0} & {\black 0} \\
&{\black 0} & {\black 0} & {\black 0} & {\black 0} & {\black 0} & {\black 0} & {\black 0} & {\black 0} & {\black 0} & {\black 0} & {\black 0} & {\black 0} & {\black 0} & {\black 0} & {\black 0} & {\black 0} & {\black 0} & {\black 0} & {\black 0} & {\black 0} & {\black 0} & {\black 0} & {\black 0} & {\black 0} & {\black 0} & {\black 0} & {\black 0} & {\black 0} & {\black 0} & {\black 0} & {\black 0} & {\black 0} & {\black 0} & {\black 0} & {\black 0} & {\black 0} & \scalebox{2}{$\mathbf{1}$} & {\black 0} & \ \scalebox{2}{$\mathbf{y}$}   & {\black 0} & {\black 0} & {\black 0} & {\black 0} & {\black 0} & {\black 0} & {\black 0} & {\black 0} & {\black 0} \\
&{\black 0} & {\black 0} & {\black 0} & {\black 0} & {\black 0} & {\black 0} & {\black 0} & {\black 0} & {\black 0} & {\black 0} & {\black 0} & {\black 0} & {\black 0} & {\black 0} & {\black 0} & {\black 0} & {\black 0} & {\black 0} & {\black 0} & {\black 0} & {\black 0} & {\black 0} & {\black 0} & {\black 0} & {\black 0} & {\black 0} & {\black 0} & {\black 0} & {\black 0} & {\black 0} & {\black 0} & {\black 0} & {\black 0} & {\black 0} & {\black 0} & {\black 0} & {\black 0} & {\black 0} & {\black 0} & \scalebox{2}{$\mathbf{1}$} & {\black 0} & \ \scalebox{2}{$\mathbf{x}$}   & {\black 0} & {\black 0} & {\black 0} & {\black 0} & {\black 0} & {\black 0} \\
&{\black 0} & {\black 0} & {\black 0} & {\black 0} & {\black 0} & {\black 0} & {\black 0} & {\black 0} & {\black 0} & {\black 0} & {\black 0} & {\black 0} & {\black 0} & {\black 0} & {\black 0} & {\black 0} & {\black 0} & {\black 0} & {\black 0} & {\black 0} & {\black 0} & {\black 0} & {\black 0} & {\black 0} & {\black 0} & {\black 0} & {\black 0} & {\black 0} & {\black 0} & {\black 0} & {\black 0} & {\black 0} & {\black 0} & {\black 0} & {\black 0} & {\black 0} & {\black 0} & {\black 0} & {\black 0} & {\black 0} & {\black 0} & {\black 0} & \scalebox{2}{$\mathbf{1}$} & {\black 0} &  \scalebox{2}{$\mathbf{\bar y}$}   & {\black 0} & {\black 0} & {\black 0} \\
\scalebox{2}{\textrm{\black Row 32:}} &{\black 0} & {\black 0} & {\black 0} & {\black 0} & {\black 0} & {\black 0} & {\black 0} & {\black 0} & {\black 0} & {\black 0} & {\black 0} & {\black 0} & {\black 0} & {\black 0} & {\black 0} & {\black 0} & {\black 0} & {\black 0} & {\black 0} & {\black 0} & {\black 0} & {\black 0} & {\black 0} & {\black 0} & {\black 0} & {\black 0} & {\black 0} & {\black 0} & {\black 0} & {\black 0} & {\black 0} & {\black 0} & {\black 0} & {\black 0} & {\black 0} & {\black 0} & {\black 0} & {\black 0} & {\black 0} & {\black 0} & {\black 0} & {\black 0} & {\black 0} & {\black 0} & {\black 0} & \scalebox{2}{$\mathbf{1}$} & {\black 0} &  \scalebox{2}{$\mathbf{\bar x}$}  \\
%%%%%%%%%%%%%%%%%%%%%%%%%%%%%%%%%%%%%%%%%%
\scalebox{2}{\textrm{\black Row 33:}} &\scalebox{2}{$\mathbf{1}$} & {\black 0} & {\black 0} & {\black 0} & {\black 0} & {\black 0} & {\black 0} & {\black 0} & {\black 0} & {\black 0} & {\black 0} & {\black 0} & {\black 0} & {\black 0} & {\black 0} & \scalebox{2}{$\mathbf{1}$} & {\black 0} & {\black 0} & {\black 0} & {\black 0} & {\black 0} & {\black 0} & {\black 0} & {\black 0} & {\black 0} & {\black 0} & {\black 0} & \scalebox{2}{$\mathbf{1}$} & {\black 0} & {\black 0} & {\black 0} & {\black 0} & {\black 0} & {\black 0} & {\black 0} & {\black 0} & {\black 0} & {\black 0} & {\black 0} & \scalebox{2}{$\mathbf{1}$} & {\black 0} & {\black 0} & {\black 0} & {\black 0} & {\black 0} & {\black 0} & {\black 0} & {\black 0} \\
&{\black 0} & {\black 0} & {\black 0} & \scalebox{2}{$\mathbf{1}$} & {\black 0} & {\black 0} & {\black 0} & {\black 0} & {\black 0} & {\black 0} & {\black 0} & {\black 0} & \scalebox{2}{$\mathbf{1}$} & {\black 0} & {\black 0} & {\black 0} & {\black 0} & {\black 0} & {\black 0} & {\black 0} & {\black 0} & {\black 0} & {\black 0} & {\black 0} & \scalebox{2}{$\mathbf{1}$} & {\black 0} & {\black 0} & {\black 0} & {\black 0} & {\black 0} & {\black 0} & {\black 0} & {\black 0} & {\black 0} & {\black 0} & {\black 0} & \scalebox{2}{$\mathbf{1}$} & {\black 0} & {\black 0} & {\black 0} & {\black 0} & {\black 0} & {\black 0} & {\black 0} & {\black 0} & {\black 0} & {\black 0} & {\black 0} \\
&{\black 0} & {\black 0} & {\black 0} & {\black 0} & {\black 0} & {\black 0} & {\black 0} & \scalebox{2}{$\mathbf{1}$} & \scalebox{2}{$\mathbf{1}$} & {\black 0} & \scalebox{2}{$\mathbf{1}$} & {\black 0} & {\black 0} & {\black 0} & {\black 0} & {\black 0} & {\black 0} & {\black 0} & {\black 0} & \scalebox{2}{$\mathbf{1}$} & \scalebox{2}{$\mathbf{1}$} & {\black 0} & \scalebox{2}{$\mathbf{1}$} & {\black 0} & {\black 0} & {\black 0} & {\black 0} & {\black 0} & {\black 0} & {\black 0} & {\black 0} & {\black 0} & {\black 0} & {\black 0} & {\black 0} & \scalebox{2}{$\mathbf{1}$} & {\black 0} & {\black 0} & {\black 0} & {\black 0} & {\black 0} & {\black 0} & {\black 0} & {\black 0} & \scalebox{2}{$\mathbf{1}$} & {\black 0} & {\black 0} & {\black 0} \\
&{\black 0} & \scalebox{2}{$\mathbf{1}$} & {\black 0} & {\black 0} & \scalebox{2}{$\mathbf{1}$} & \scalebox{2}{$\mathbf{1}$} & {\black 0} & {\black 0} & {\black 0} & {\black 0} & {\black 0} & {\black 0} & {\black 0} & \scalebox{2}{$\mathbf{1}$} & {\black 0} & {\black 0} & \scalebox{2}{$\mathbf{1}$} & \scalebox{2}{$\mathbf{1}$} & {\black 0} & {\black 0} & {\black 0} & {\black 0} & {\black 0} & {\black 0} & {\black 0} & {\black 0} & \scalebox{2}{$\mathbf{1}$} & {\black 0} & {\black 0} & {\black 0} & {\black 0} & {\black 0} & {\black 0} & {\black 0} & {\black 0} & {\black 0} & {\black 0} & {\black 0} & {\black 0} & {\black 0} & {\black 0} & \scalebox{2}{$\mathbf{1}$} & {\black 0} & {\black 0} & {\black 0} & {\black 0} & {\black 0} & {\black 0} \\
&{\black 0} & {\black 0} & \scalebox{2}{$\mathbf{1}$} & {\black 0} & {\black 0} & {\black 0} & {\black 0} & {\black 0} & {\black 0} & {\black 0} & {\black 0} & {\black 0} & {\black 0} & {\black 0} & \scalebox{2}{$\mathbf{1}$} & {\black 0} & {\black 0} & {\black 0} & {\black 0} & {\black 0} & {\black 0} & {\black 0} & {\black 0} & {\black 0} & {\black 0} & \scalebox{2}{$\mathbf{1}$} & {\black 0} & {\black 0} & \scalebox{2}{$\mathbf{1}$} & \scalebox{2}{$\mathbf{1}$} & {\black 0} & {\black 0} & {\black 0} & {\black 0} & {\black 0} & {\black 0} & {\black 0} & \scalebox{2}{$\mathbf{1}$} & \scalebox{2}{$\mathbf{1}$} & {\black 0} & \scalebox{2}{$\mathbf{1}$} & {\black 0} & {\black 0} & {\black 0} & {\black 0} & {\black 0} & {\black 0} & {\black 0} \\
& {\black 0} & {\black 0} & {\black 0} & {\black 0} & {\black 0} & {\black 0} & {\black 0} & {\black 0} & {\black 0} & {\black 0} & {\black 0} & \scalebox{2}{$\mathbf{1}$} & {\black 0} & {\black 0} & {\black 0} & {\black 0} & {\black 0} & {\black 0} & {\black 0} & {\black 0} & {\black 0} & {\black 0} & {\black 0} & \scalebox{2}{$\mathbf{1}$} & {\black 0} & {\black 0} & {\black 0} & {\black 0} & {\black 0} & {\black 0} & {\black 0} & \scalebox{2}{$\mathbf{1}$} & \scalebox{2}{$\mathbf{1}$} & {\black 0} & \scalebox{2}{$\mathbf{1}$} & {\black 0} & {\black 0} & {\black 0} & {\black 0} & {\black 0} & {\black 0} & {\black 0} & {\black 0} & \scalebox{2}{$\mathbf{1}$} & {\black 0} & {\black 0} & \scalebox{2}{$\mathbf{1}$} & \scalebox{2}{$\mathbf{1}$} \\
&{\black 0} & {\black 0} & {\black 0} & {\black 0} & {\black 0} & {\black 0} & \scalebox{2}{$\mathbf{1}$} & {\black 0} & {\black 0} & {\black 0} & {\black 0} & {\black 0} & {\black 0} & {\black 0} & {\black 0} & {\black 0} & {\black 0} & {\black 0} & {\black 0} & {\black 0} & {\black 0} & \scalebox{2}{$\mathbf{1}$} & {\black 0} & {\black 0} & {\black 0} & {\black 0} & {\black 0} & {\black 0} & {\black 0} & {\black 0} & \scalebox{2}{$\mathbf{1}$} & {\black 0} & {\black 0} & {\black 0} & {\black 0} & {\black 0} & {\black 0} & {\black 0} & {\black 0} & {\black 0} & {\black 0} & {\black 0} & \scalebox{2}{$\mathbf{1}$} & {\black 0} & {\black 0} & {\black 0} & {\black 0} & {\black 0} \\
&{\black 0} & {\black 0} & {\black 0} & {\black 0} & {\black 0} & {\black 0} & {\black 0} & {\black 0} & {\black 0} & \scalebox{2}{$\mathbf{1}$} & {\black 0} & {\black 0} & {\black 0} & {\black 0} & {\black 0} & {\black 0} & {\black 0} & {\black 0} & \scalebox{2}{$\mathbf{1}$} & {\black 0} & {\black 0} & {\black 0} & {\black 0} & {\black 0} & {\black 0} & {\black 0} & {\black 0} & {\black 0} & {\black 0} & {\black 0} & {\black 0} & {\black 0} & {\black 0} & \scalebox{2}{$\mathbf{1}$} & {\black 0} & {\black 0} & {\black 0} & {\black 0} & {\black 0} & {\black 0} & {\black 0} & {\black 0} & {\black 0} & {\black 0} & {\black 0} & \scalebox{2}{$\mathbf{1}$} & {\black 0} & {\black 0} \\
&{\black 0} & {\black 0} & \scalebox{2}{$\mathbf{1}$} & {\black 0} & \scalebox{2}{$\mathbf{1}$} & {\black 0} & \scalebox{2}{$\mathbf{1}$} & {\black 0} & {\black 0} & {\black 0} & {\black 0} & {\black 0} & {\black 0} & {\black 0} & \scalebox{2}{$\mathbf{1}$} & {\black 0} & \scalebox{2}{$\mathbf{1}$} & {\black 0} & \scalebox{2}{$\mathbf{1}$} & {\black 0} & {\black 0} & {\black 0} & {\black 0} & {\black 0} & {\black 0} & {\black 0} & {\black 0} & {\black 0} & {\black 0} & {\black 0} & {\black 0} & {\black 0} & {\black 0} & {\black 0} & {\black 0} & {\black 0} & {\black 0} & {\black 0} & {\black 0} & {\black 0} & {\black 0} & {\black 0} & {\black 0} & {\black 0} & {\black 0} & {\black 0} & {\black 0} & {\black 0} \\
&{\black 0} & \scalebox{2}{$\mathbf{1}$} & {\black 0} & {\black 0} & {\black 0} & \scalebox{2}{$\mathbf{1}$} & {\black 0} & {\black 0} & {\black 0} & \scalebox{2}{$\mathbf{1}$} & {\black 0} & {\black 0} & {\black 0} & {\black 0} & {\black 0} & {\black 0} & {\black 0} & {\black 0} & {\black 0} & {\black 0} & {\black 0} & {\black 0} & {\black 0} & {\black 0} & {\black 0} & \scalebox{2}{$\mathbf{1}$} & {\black 0} & {\black 0} & {\black 0} & \scalebox{2}{$\mathbf{1}$} & {\black 0} & {\black 0} & {\black 0} & \scalebox{2}{$\mathbf{1}$} & {\black 0} & {\black 0} & {\black 0} & {\black 0} & {\black 0} & {\black 0} & {\black 0} & {\black 0} & {\black 0} & {\black 0} & {\black 0} & {\black 0} & {\black 0} & {\black 0} \\
& \scalebox{2}{$\mathbf{1}$} & {\black 0} & {\black 0} & {\black 0} & {\black 0} & {\black 0} & {\black 0} & {\black 0} & \scalebox{2}{$\mathbf{1}$} & {\black 0} & \scalebox{2}{$\mathbf{1}$} & {\black 0} & {\black 0} & {\black 0} & {\black 0} & {\black 0} & {\black 0} & {\black 0} & {\black 0} & {\black 0} & {\black 0} & {\black 0} & {\black 0} & {\black 0} & {\black 0} & {\black 0} & {\black 0} & {\black 0} & {\black 0} & {\black 0} & {\black 0} & {\black 0} & {\black 0} & {\black 0} & {\black 0} & {\black 0} & {\black 0} & {\black 0} & {\black 0} & \scalebox{2}{$\mathbf{1}$} & {\black 0} & {\black 0} & {\black 0} & \scalebox{2}{$\mathbf{1}$} & {\black 0} & {\black 0} & {\black 0} & \scalebox{2}{$\mathbf{1}$} \\
&{\black 0} & {\black 0} & {\black 0} & \scalebox{2}{$\mathbf{1}$} & {\black 0} & {\black 0} & {\black 0} & \scalebox{2}{$\mathbf{1}$} & {\black 0} & {\black 0} & {\black 0} & \scalebox{2}{$\mathbf{1}$} & {\black 0} & {\black 0} & {\black 0} & \scalebox{2}{$\mathbf{1}$} & {\black 0} & {\black 0} & {\black 0} & \scalebox{2}{$\mathbf{1}$} & {\black 0} & {\black 0} & {\black 0} & \scalebox{2}{$\mathbf{1}$} & {\black 0} & {\black 0} & {\black 0} & {\black 0} & {\black 0} & {\black 0} & {\black 0} & {\black 0} & {\black 0} & {\black 0} & {\black 0} & {\black 0} & {\black 0} & {\black 0} & {\black 0} & {\black 0} & {\black 0} & {\black 0} & {\black 0} & {\black 0} & {\black 0} & {\black 0} & {\black 0} & {\black 0} \\
&{\black 0} & {\black 0} & {\black 0} & {\black 0} & {\black 0} & {\black 0} & {\black 0} & {\black 0} & {\black 0} & {\black 0} & {\black 0} & {\black 0} & {\black 0} & \scalebox{2}{$\mathbf{1}$} & {\black 0} & {\black 0} & {\black 0} & \scalebox{2}{$\mathbf{1}$} & {\black 0} & {\black 0} & {\black 0} & \scalebox{2}{$\mathbf{1}$} & {\black 0} & {\black 0} & {\black 0} & {\black 0} & {\black 0} & {\black 0} & {\black 0} & {\black 0} & {\black 0} & {\black 0} & {\black 0} & {\black 0} & {\black 0} & {\black 0} & {\black 0} & {\black 0} & \scalebox{2}{$\mathbf{1}$} & {\black 0} & \scalebox{2}{$\mathbf{1}$} & {\black 0} & \scalebox{2}{$\mathbf{1}$} & {\black 0} & {\black 0} & {\black 0} & {\black 0} & {\black 0} \\
&{\black 0} & {\black 0} & {\black 0} & {\black 0} & {\black 0} & {\black 0} & {\black 0} & {\black 0} & {\black 0} & {\black 0} & {\black 0} & {\black 0} & \scalebox{2}{$\mathbf{1}$} & {\black 0} & {\black 0} & {\black 0} & {\black 0} & {\black 0} & {\black 0} & {\black 0} & \scalebox{2}{$\mathbf{1}$} & {\black 0} & \scalebox{2}{$\mathbf{1}$} & {\black 0} & \scalebox{2}{$\mathbf{1}$} & {\black 0} & {\black 0} & {\black 0} & {\black 0} & {\black 0} & {\black 0} & {\black 0} & \scalebox{2}{$\mathbf{1}$} & {\black 0} & \scalebox{2}{$\mathbf{1}$} & {\black 0} & {\black 0} & {\black 0} & {\black 0} & {\black 0} & {\black 0} & {\black 0} & {\black 0} & {\black 0} & {\black 0} & {\black 0} & {\black 0} & {\black 0} \\
&{\black 0} & {\black 0} & {\black 0} & {\black 0} & {\black 0} & {\black 0} & {\black 0} & {\black 0} & {\black 0} & {\black 0} & {\black 0} & {\black 0} & {\black 0} & {\black 0} & {\black 0} & {\black 0} & {\black 0} & {\black 0} & {\black 0} & {\black 0} & {\black 0} & {\black 0} & {\black 0} & {\black 0} & {\black 0} & {\black 0} & \scalebox{2}{$\mathbf{1}$} & {\black 0} & \scalebox{2}{$\mathbf{1}$} & {\black 0} & \scalebox{2}{$\mathbf{1}$} & {\black 0} & {\black 0} & {\black 0} & {\black 0} & {\black 0} & {\black 0} & \scalebox{2}{$\mathbf{1}$} & {\black 0} & {\black 0} & {\black 0} & \scalebox{2}{$\mathbf{1}$} & {\black 0} & {\black 0} & {\black 0} & \scalebox{2}{$\mathbf{1}$} & {\black 0} & {\black 0} \\
\scalebox{2}{\textrm{\black Row 48:}} &{\black 0} & {\black 0} & {\black 0} & {\black 0} & {\black 0} & {\black 0} & {\black 0} & {\black 0} & {\black 0} & {\black 0} & {\black 0} & {\black 0} & {\black 0} & {\black 0} & {\black 0} & {\black 0} & {\black 0} & {\black 0} & {\black 0} & {\black 0} & {\black 0} & {\black 0} & {\black 0} & {\black 0} & {\black 0} & {\black 0} & {\black 0} & \scalebox{2}{$\mathbf{1}$} & {\black 0} & {\black 0} & {\black 0} & \scalebox{2}{$\mathbf{1}$} & {\black 0} & {\black 0} & {\black 0} & \scalebox{2}{$\mathbf{1}$} & \scalebox{2}{$\mathbf{1}$} & {\black 0} & {\black 0} & {\black 0} & {\black 0} & {\black 0} & {\black 0} & {\black 0} & \scalebox{2}{$\mathbf{1}$} & {\black 0} & \scalebox{2}{$\mathbf{1}$} & {\black 0}
\end{array}}
%\right)
$}}
\end{center}
\caption{The matrix $J$, where ${\rm x} = \frac{9+i\sqrt{15}}{8}$ and ${\rm y}=-\frac{3+i\sqrt{15}}{4}$.}
\label{table:J}
\end{table}

\section{A parametrisation of $X_0$}\label{s:param}
As stated in the introduction, an actual parametrisation of a family of representations is a crucial tool for constructing 
geometric structures. Among other things, it allows explicit constructions of fundamental domains \cite{F, DF8, ParkWi2}.

We describe in this section a parametrisation of a Zariski open subset of $X_0$ by actual matrices. 
More precisely, given four generic complex numbers $z_1$, $z_2$, $z_3$ and $z_4$ we are going to provide two 
pairs $(A,B)$ of regular order three elements in SL(3,$\C$) such that
\begin{itemize}
 \item  $A$ and $B$ have no common eigenvector in $\C^3$,
 \item  $A$ and $B$ satisfy
 \begin{equation}\label{eq:param} \tr (AB)=z_1,\, \tr (A^{-1}B)=z_2,\, \tr (A^{-1}B^{-1})=z_3,\, \tr (AB^{-1})=z_4.
 \end{equation}
\end{itemize}
Recall that the traces of $A$, $B$ and their inverses are zero as they are regular order three elements.
The genericity condition will be made explicit in Proposition \ref{pr:paramX0}.

We know from Lawton's theorem that $X_0$ is a double cover of $\C^4$. Parameters on $\C^4$ are 
given by the four trace parameters $z_1$, $z_2$, $z_3$, $z_4$. 
To describe this double cover, define the quantity
$$\Delta = z_{1}^{2} z_{3}^{2} - 2 z_{1} z_{2} z_{3} z_{4} + z_{2}^{2} z_{4}^{2} -
 4 z_{1}^{3} - 4 z_{2}^{3} - 4 z_{3}^{3} - 4 z_{4}^{3} + 18 z_{1} z_{3} + 18 z_{2} z_{4} - 27,$$
which is the discriminant of the trace equation \eqref{trace-equation} in Theorem \ref{theo-lawton}, in the case where the traces of $A$, $B$ and their inverses vanish.
We denote by $\delta$ a square root of $\Delta$. Let $\om$ be a non trivial cube root of $1$ and 
$k = \mathbb Q[\om](z_1,z_2,z_3,z_4)$.  

\begin{prop}\label{pr:paramX0}
Let $(a,b,c,d)$ be the following elements in $k[\delta]$:
\begin{eqnarray*}
 4a & = & \dfrac{z_1z_3-z_2z_4+6\om z_1+6\om ^2 z_3+9+\delta}{\om z_1+z_2+\om^2z_3+z_4+3}\nonumber\\
 4d & = & \dfrac{z_1z_3-z_2z_4+6\om^2z_1+6\om z_3+9-\delta}{\om^2 z_1+z_2+\om z_3+z_4+3}\nonumber\\
 b & = & \frac{z_1 - z_2-\om^2 z_3 + \om^2 z_4+3(\om^2-1)}{z_1 + z_2 + \om^2 z_3+\om^2 z_4+3\om} a 
 +(\om-1)\frac{z_1 + \om z_2 + \om z_3 + z_4+3\om^2}{z_1 + z_2 + \om^2 z_3+\om^2 z_4+3\om}\nonumber\\
 c & = & \frac{z_1 + \om z_2-\om z_3 - z_4+3(\om-1)}{z_1 + \om z_2 +  \om z_3+ z_4+3\om^2} d
 +(\om^2-1)\dfrac{z_1 + z_2 + \om^2 z_3 + \om^2z_4+3\om}{z_1 + \om z_2 + \om z_3 + z_4+3\om^2}.\nonumber\\
\end{eqnarray*}
Then for any 4-tuple of complex numbers $(z_1,z_2,z_3,z_4)$ such that $a$, $b$ $c$ and $d$ are well-defined, any pair 
$(A,B)$ of order three regular elements  of {\rm SL($3$,$\C$)},\, satisfying \eqref{eq:param} and such that the $\om^2$-eigenline 
for $A$ is different from the $\om$-eigenline for $B$,\,  is conjugate in {\rm SL(3,$\C$)} to one of the two pairs defined by
\begin{equation}\label{normalAB}
A=\begin{bmatrix} \om & 0 & 0 \\ \om^2 & 1 & 0 \\ b +a & 2\om a & \om^2  \end{bmatrix}\mbox{ and }
B = \begin{bmatrix}\om & 2\om^2d & c+d \\ 0 & 1 & \om\\ 0 & 0 & \om^2 \end{bmatrix}.
\end{equation}
\end{prop}

 It should be noted here that it is necessary that the $\omega$ and $\omega^2$ eigenlines for respectively $A$ and $B$ are 
disjoint to obtain a normalisation such as \eqref{normalAB}. This condition is always satisfied when the pair $(A,B)$ is 
irreducible, that is if the group $\la A,B\ra$ doesn't have a global fixed point in its action on $\C P^2$. 

\begin{proof}
A direct computation of the traces of $AB$, $A^{-1}B$, $A^{-1}B^{-1}$ and $AB^{-1}$ with
the given values leads to a verification of our parametrization \cite{Notebook}. We now indicate how to obtain 
these values.

Recall first that regular order three elements in SL(3,$\C$) have eigenvalue spectrum $\{1,\om,\om^2\}$.
First, one may conjugate the pair $(A,B)$ so that $A$ and $B$ are respectively lower and upper triangular, with 
eigenvalues organised as in \eqref{normalAB}. This amounts to choosing a 
basis of $\C^3$ of the form $(v_B,v,v_A)$ where $v_A$ (resp. $v_B$) is a $\om^2$-eigenvector for $A$ (resp. a 
$\om$-eigenvector for $B$), and $v$ is a non zero vector in the intersection $V_A\cap V_B$, where $V_A$  (resp. $V_B$) 
is spanned by $v_A$ and a $1$-eigenvector of $A$ (resp. $v_B$ and a $1$ eigenvector for $B$). Conjugating by a diagonal 
matrix allows to bring off-diagonal coefficients equal to $\om^2$ in $A$ and to $\om$ in $B$ as shown in \eqref{normalAB}.

We need now to determine $a,b,c,$ and $d$ from \eqref{eq:param}. These four conditions correspond to the following 
system of equations. 

\begin{align}[left= \bigl(\Sigma\bigr)\quad \empheqlbrace]
 (a+b)(d+c) +2a\om^2+2\om d & = z_1 \label{eqAB} \\
 (a-b)(d+c) -2a-2d+3& =  z_2 \label{eqAmB} \\
 (a-b)(d-c) +2\om a+2\om^2d& = z_3 \label{eqAmBm} \\
 (a+b)(d-c) -2a-2d+3& = z_4 \label{eqABm}
\end{align}

This system is relatively easy to solve using a computer and, for instance, Gr\"obner bases. 
However, it is also solvable by hand, and we indicate now how to do it. Before going any further, let us observe that 
conjugating the pair $(A,B)$ by the matrix
$$\begin{bmatrix} 0 & 0 & 1\\ 0 & 1 & 0\\ 1 & 0 & 0\end{bmatrix}$$
amounts to do following exchanges in $\Sigma$:
\begin{equation} a\longleftrightarrow d,\,b\longleftrightarrow c,\, \om\longleftrightarrow \om^2 \label{echange}.\end{equation}
More precisely, the left-hand sides of \eqref{eqAB} and \eqref{eqAmBm} are preserved by these changes, whereas the 
left-hand sides of \eqref{eqAmB} and \eqref{eqABm} are exchanged. 

Next, we compute linear combinations of the above equations, and obtain the following equivalent system:
\begin{align}[left= \bigl(\Sigma'\bigr)\empheqlbrace]
\eqref{eqAB}+\eqref{eqAmB}+\eqref{eqAmBm}+\eqref{eqABm} & :\quad 4ad-6a-6d-z_1-z_2-z_3-z_4+6=0\label{e1}\\
\eqref{eqAB}-\eqref{eqAmB}+\eqref{eqAmBm}-\eqref{eqABm} & :\quad 4bc+2a+2d-z_1+z_2-z_3+z_4-6=0\label{e2}\\
\eqref{eqAB}-\eqref{eqAmBm} & :\quad  2ac+2bd-4a\om+4d\om-2a+2d-z_1+z_3=0\label{e3}\\
\eqref{eqAmB}-\eqref{eqABm} & :\quad  2ac-2bd-z_2+z_4=0\label{e4}
\end{align}
To obtain the value of $a$ announced in the statement, we proceed as follows.
\begin{itemize}
 \item The two equations \eqref{e3} and \eqref{e4} are linear in $b$ and $c$. We solve them to obtain expressions of $b$ and $c$ in 
terms of $a$ and $d$.
 \item Plugging these expressions of $b$ and $c$ in \eqref{e2} and taking numerator, we obtain an equation that relates 
$a$ and $d$ and involves the monomials $a^2d$, $ad^2$, $a^2$, $d^2$, $ad$, $a$ and $d$. This equation can be simplified by observing 
that the product $ad$ can be expressed as an affine function of $a$ and $d$ using \eqref{e1}. Doing so, most of the monomials simplify and we obtain a linear relation between $a$ and $d$. This yields an expression of $d$ as a function of $a$, which can be inserted back in \eqref{e1}. 
We obtain this way a quadratic equation in $a$, which is:

\begin{align}\label{quadratics-a}
0  =  &  4(\om z_1+z_2+\om^2z_3+z_4+3)a^2 - 2(6\om z_1+6\om^2z_3+z_1z_3-z_2z_4+9)a \\
& +9+6\om z_1-3z_2+6\om^2z_3-3z_4  +\om^2z_1^2+z_2^2+\om z_3^2+z_4^2  \nonumber\\
&  -\om z_1z_2+2z_1z_3-\om z_1z_4-\om^2z_2z_3-z_2z_4-\om^2z_3z_4 \nonumber
\end{align}
\end{itemize}
The discriminant of this quadratic equation is $\Delta$, and we obtain two possible values for $a$, corresponding to
the two square roots of $\Delta$. We obtain in turn the value of $d$ given in the 
statement. Note that $d$ is obtained from $a$ by the exchanges $\om\longleftrightarrow \om^2$ and $z_2\longleftrightarrow z_4$. 
This correspond to the symmetry of the system $(\Sigma)$ given in \eqref{echange}.

As observed above, knowing the values of $a$ and $d$ gives us the values of $b$ and $c$. However, the expressions obtained by solving
\eqref{e3} and \eqref{e4} are not exactly those given in the statement, and simplifying them is quite intricate. 
The following strategy gives a way to determine $b$ and $c$ more directly. 
First of all, we know from Lawton's theorem and our determination of $a$ and $d$ that $b$ 
belongs to $k[\delta]$ (recall that $k=\mathbb Q[\om](z_1,z_2,z_3,z_4)$). 
Hence, we may look for it under the form $b = aP+Q$ where $P$ and $Q$ belong to 
$k$. We use this form in the equation \eqref{e3}$-$\eqref{e4},
and also plug the values of $a$ and $d$. This leads to an equation, linear in $P$ and $Q$, 
between two elements of $k[\delta]$. Isolating the coefficient of $\delta$
and the remaining part, we find two linear equations in $P$ and $Q$.  Solving those equations leads to the given 
value for $b$. The value for $c$ can be obtained using the symmetries of $\Sigma$ given in \eqref{echange}.
\end{proof}

\bibliographystyle{amsalpha}\bibliography{biblio}
\end{document}